\documentclass[a4paper,10pt,twoside]{article}
\usepackage[a4paper,left=2.2cm,right=2.2cm,top=2.3cm,bottom=2.6cm]{geometry}
\usepackage[english]{babel}
\usepackage{setspace}

\usepackage{mathrsfs,hyperref}
\usepackage{harvard} 
\usepackage[]{amsmath}
\usepackage{amsthm}
\usepackage{fix-cm}
\usepackage[]{amssymb}
\usepackage[]{latexsym}
\usepackage[latin1]{inputenc}
\usepackage[right]{eurosym}
\usepackage[T1]{fontenc}
\usepackage[]{graphicx}
\usepackage[]{epsfig}
\usepackage{fancyhdr}
\usepackage{bbm}
\makeatletter

\renewcommand{\p@enumi}{theenumi-}
\renewcommand{\@fnsymbol}[1]{\@alph{#1}}

\newcommand{\bbr}{\mathbb{R}}
\newcommand{\bbe}{\mathbb{E}}
\newcommand{\bbn}{\mathbb{N}}
\newcommand{\bbp}{\mathbb{P}}

\newcommand{\bbg}{\mathbb{G}}
\newcommand{\bbf}{\mathbb{F}}

\newcommand{\ci}{\citeasnoun}

\newcommand{\fil}{\mathcal{F}}
\newcommand{\fcal}{\mathcal{F}}
\newcommand{\gcal}{\mathcal{G}}

\newcommand{\hcal}{\mathcal{H}}

\newcommand{\pcal}{\mathcal{P}}
\newcommand{\ecal}{\mathcal{E}}
\newcommand{\bcal}{\mathcal{B}}

\newcommand{\lcal}{\mathcal{L}}

\newcommand{\ccal}{\mathcal{C}}

\newcommand{\acal}{\mathcal{A}}

\newcommand\independent{\protect\mathpalette{\protect\independenT}{\perp}}
\def\independenT#1#2{\mathrel{\rlap{$#1#2$}\mkern2mu{#1#2}}}

\makeatother
\newtheorem{lemma}{Lemma}[section]
\newtheorem{proposition}[lemma]{Proposition}
\newtheorem{theorem}[lemma]{Theorem}

\newtheorem{definition}[lemma]{Definition}
\newtheorem{example}[lemma]{Example}

\newtheorem{remark}[lemma]{Remark}
\newtheorem{assumption}[lemma]{Assumption}

\begin{document}

\title{Stochastic Mortality Models: \\ An Infinite-Dimensional Approach}
\author{Stefan Tappe ~~~~~~  Stefan Weber\\[0.7ex] \textit{Leibniz Universit{\"a}t Hannover \thanks{Leibniz Universit\"at Hannover, Institut f\"ur Mathematische Stochastik, Welfengarten 1,
30167 Hannover, Germany. E-mail: Stefan Tappe {\tt <tappe$@$stochastik.uni-hannover.de>},  Stefan Weber {\tt <sweber$@$stochastik.uni-hannover.de>}. }} }
\date{}
\maketitle

\begin{abstract}
Demographic projections of future mortality rates involve a high level of uncertainty and require stochastic mortality models. The current paper investigates forward mortality models driven by a (possibly infinite dimensional) Wiener process and a compensated Poisson random measure. A major innovation of the paper is the introduction of a family of processes called forward mortality improvements which provide a flexible tool for a simple construction of stochastic forward mortality models. In practice, the notion of mortality improvements are a convenient device for the quantification of  changes in mortality rates over time that enables, for example, the detection of cohort effects. 

We show that the forward mortality rates satisfy Heath-Jarrow-Morton-type consistency conditions which translate to the forward mortality improvements. While the consistency conditions of the forward mortality rates are analogous to the classical conditions in the context of bond markets, the conditions of the forward mortality improvements possess a different structure: forward mortality models include a cohort parameter besides the time horizon; these two dimensions are coupled in the dynamics of consistent models of forwards mortality improvements. In order to obtain a unified framework, we transform the systems of It\^o-processes which describe the forward mortality rates and improvements: in contrast to term-structure models, the corresponding stochastic partial differential equations (SPDEs) describe the random dynamics of two-dimensional surfaces rather than curves.
\end{abstract}

\vspace{0.2cm}

\noindent  \\ \vspace{-0.2cm}
 \vspace{-0.3cm}

\section{Introduction}
\label{intro}
Actuarial mathematics is often a pragmatic and simplified approach to reality.  Classical life insurance mathematics, for example, is concerned with the valuation of insurance products and the computation of reserves and is based on the idea of pooling (formalized by the equivalence principle). It requires the availability of suitable projections of the mortality of the insured individuals. In practice, insurance companies typically use deterministic mortality tables which are constructed from past mortality data and include safety margins. For standard life insurance products like term life insurance or annuities, projections are needed that stretch over various decades. 

While deterministic mortality tables constitute a rather substantial simplification of reality, insurers and actuaries were always aware of the fact that demographic projections of future mortality rates involve a high level of uncertainty. Instead of attempting to correctly predict future mortality rates, actuarial practice implements a prudent risk management scheme that involves substantial safety margins. While customers of life insurance companies are typically overcharged, the mechanism is fair in the sense that surpluses are redistributed to the insured. Actuarial life tables must not be interpreted as models of actual mortality rates, but rather as specific technical tools inside actuarial practice.

In reality, demographic projections of future mortality rates involve a high level of uncertainty. This has, for example, been observed by \ci{booth}. Understanding mortality is thus not the same as analyzing or constructing insurance life tables and requires stochastic models of mortality rates and of mortality projection mechanisms. The current article focuses on stochastic forward mortality models (see e.g. \ci{milevsky}, \ci{dahl}, \ci{miltersen}, \ci{cairns}, \ci{Bauer}, \ci{Barbarin}, \ci{norberg}, \ci{BBK}, \ci{zhu-bauer-geneva}, \ci{zhu-bauer}). These are closely related to intensity models of mortality which were discussed by \ci{biffis}, \ci{biffis2}, \ci{hainaut}, \ci{luciano}, and  \ci{schrager}, among others. Our main contribution is to provide a mathematically rigorous and transparent framework for this approach that generalizes and substantially clarifies previous contributions in the literature. 

Stochastic mortality and mortality projection models can be employed as a  framework for analyzing the reliability, robustness and cost of current actuarial practice. They might also improve demographic projections and provide a better basis for the management of mortality and longevity risk. Finally, stochastic mortality models are needed for the computation of the market-consistent value of insurance liabilities -- a quantity that is particularly important for management and reporting purposes. Stochastic mortality and mortality projection models constitute also essential ingredients for the construction of various re-insurance or capital market solutions that facilitate the mitigation of mortality and longevity risk. Recent product innovations include mortality swaps, longevity bonds, and q-forwards.
 
\paragraph{Contribution and Outline:}
 
The current paper investigates stochastic forward mortality models. Section~\ref{DEP} provides a motivation for our approach by introducing a dynamic point process model of the mortality of individuals (cf. \ci{bremaud}, \ci{BR02}). We define forward mortality processes and rates. A major innovation of the paper is the introduction of a family of processes called forward mortality improvements which constitute a flexible tool that enables a simple construction of stochastic forward mortality models. Additionally, in practice, the notion of mortality improvements provides a convenient instrument to quantify changes in mortality rates over time that allows, for example, the detection of cohort effects (\ci{HannoverRe}). 

Section~\ref{sec:lln} provides conditional laws of large numbers as a rational for the significance of forward mortality models. Although implicit in many papers on the subject, to the best of our knowledge, these theorems have never  rigorously been proven in the literature. A special case of forward mortality models are intensity-based models which allow an alternative description of their probabilistic dynamics via compensators. This point of view (that is preferred by some authors to the approach taken in the current paper) is explained in Section~\ref{sec:comp}. 

The suggested forward mortality models can either be interpreted as describing the real-world (see e.g. \ci{zhu-bauer}) or risk-neutral dynamics (see e.g. \ci{biffis}, \ci{BiffisMil}, \ci{biffis2}). In both cases, we can identify a martingale condition, see Remark~\ref{rem-martingale} below, that implies consistency conditions on the dynamics of any \emph{`code-book'} in the sense of \ci{carmona}. Section~\ref{IDF} describes consistency conditions for the forward mortality rates and improvements. We also present a unified modeling framework on the basis of SPDEs. The proofs of the results are deferred to the Appendix. 

While the consistency conditions of the forward mortality rates are analogous to the classical conditions in the context of bond markets, the conditions of the forward mortality improvements possess a different structure: forward mortality models include a cohort parameter besides the time horizon; these two dimensions are coupled in the dynamics of consistent models of forwards mortality improvements.

In order to obtain a unified framework, we transform the systems of It\^o-processes which describe the forward mortality rates and forward mortality improvements. In contrast to term-structure models, the corresponding stochastic partial differential equations (SPDEs) describe the random dynamics of two-dimensional surfaces rather than curves. These surfaces are parametrized by cohorts and time horizons (see also \ci{BiffisMil} for a related study in the context of random fields). Most interesting are consistent models of forward mortality improvements which induce stochastic forward mortality models. Moreover, the shape of the forward mortality surfaces requires Hilbert spaces of functions that are not covered by the literature on interest rate models, see Definition~\ref{def-space} and Example~\ref{ex-space}.

Our results are illustrated in the context of a L\'evy driven Gompertz-Makeham model of forward mortality in Section~\ref{sec-example}.

\section{Definitions and Elementary Properties}\label{DEP}

 We start by introducing the basic notions of our stochastic mortality model. We denote by $(\Omega, \gcal, \bbp)$ a sufficiently rich probability space on which all random variables and processes are defined. The probability measure $\bbp$ can either be interpreted as the real-world measure or a pricing measure depending on the context in which the theory will be applied.  
 
 The lifetime of an individual is characterized by its date of birth $c$ and its random time of death. As common in life insurance mathematics, we encode the cohorts of individuals by their age and thus define $x=-c\in\bbr$ which can be interpreted as the (hypothetical) age at time $0$. The death of the individual occurs at a  $\gcal$-measurable random time $\tau^x: \Omega \to (-x, \infty)$. Equivalently, the time of death is described by the survival indicator, a stochastic process that is defined by 
$$  N_t(x) (\omega) = \left\{\begin{array}{ll}  1, &  t < \tau ^x (\omega) , \\ 0, & t\geq  \tau ^x (\omega)  \end{array} \right.  \quad \quad \quad(\omega \in \Omega, t\in\bbr_+) . $$ 

 An established approach to modeling the probabilistic evolution of the occurrence of events are intensity-based models. A particularly convenient case are Cox process models, see \ci{bremaud}, that assume that intensities are driven by stochastic covariates. In the current paper, we focus on such models, but introduce these on a slightly more abstract level as described in \ci{bremaud} and \ci{BR02}. For a detailed analysis of filtration enlargements and further references see also \ci{jeanblanc}.

We restrict attention to the time period $\bbr_+ = [0,\infty)$.  \emph{Systematic information} is modeled by 
a filtration $\bbf = (\fcal_t)_{t \in \bbr_+}$ -- which is sometimes also called the \emph{background information}. In a classical Cox process model this filtration corresponds to the family of sigma-algebras that are generated by the history of the stochastic covariate processes. Intuitively, $\bbf$ contains all information that determines the likelihood of death events. As we will see below, it shall, however, be assumed that $\bbf$ does not include information about the exact times of death of specific individuals. For technical reasons, we assume that $(\Omega, \gcal, \bbf, \bbp)$ satisfies the usual conditions. 

\begin{itemize}
\item The conditional probability that an individual born at time $-x$ survives until time $t\geq 0$ given the background information is described by the  \emph{$\bbf$-survival process} of $\tau^x$:
\begin{align*}
G_t(t,x) := \bbp ( \tau^x > t \,|\, \fil_t ) , \quad t \geq -x\vee 0 .
\end{align*} 
This process is known to be an important ingredient in the theory of point processes. Under suitable technical assumptions, $G_t(t,x)$ equals the fraction of individuals born at date $-x$ that survive until date $t$. The precise result will be stated in Theorem~\ref{LLN}.

\item Another object of particular interest for firms that are exposed to mortality and longevity risk -- in particular for pension funds and reinsurance companies, cf. \ci{HannoverRe} -- is the best prediction at date $t$ of the fraction of individuals born at date $-x$ that survive until a future date $T$. Again, in Theorem~\ref{LLN} we will show under suitable technical conditions that this prediction equals  the \emph{$\bbf$-forward survival process} that we define as 
\begin{align}\label{def-G-official}
G_t (T,x)  :=  \bbp ( \tau^x > T \,|\, \fil_t )  = \bbe [G_T(T,x) \, | \, \fil_t ], \quad T\geq -x\vee t.
\end{align} 
\end{itemize}

\begin{remark}\label{rem-martingale}
It is apparent from its definition that, for fixed $x$ and $T$, the forward survival process $(G_t(T,x))_t$ is a martingale with respect to the probability measure $\bbp$. This martingale property is very natural:
\begin{itemize}
\item If $\bbp$ is the real-world measure, then the random variable $G_t(T,x)$ describes the conditional probability that an individual born at date $-x$ survives until date $T$ given the information at date $t$; if the available information grows as time $t$ increases, the corresponding process of such conditional probabilities is, of course, a $\bbf$-martingale. A detailed discussion of this point of view can be found in \ci{zhu-bauer}.

In addition, Theorem~\ref{LLN} provides a slightly different interpretation of the process. As mentioned before, under technical conditions, $G_t(T,x)$ is the best time-$t$-prediction of the fraction of individuals born at date $-x$ that survive until the future date $T$ (the survival ratio). If $t$ increases, the available information grows which apparently implies the martingale property.
\item If the money market account is chosen as the num\'{e}raire with deterministic interest rates $r(t)$, $t\geq 0$, and $\bbp$ is a pricing measure, then $\exp \left( -  \int_t^T r(s) ds   \right) \cdot G_t(T,x)$ can be interpreted as the price at time $t$ of a survivor bond that pays at time $T$ an amount equal to the survival ratio.
\end{itemize}
\end{remark}

 Standard technical tools from the theory of point processes are hazard processes and intensities. We adapt these notions in the context of mortality models. We start with a technical assumption.
\begin{assumption}\label{StAss0}~
 For all $t\in \bbr_+$, $T\geq  -x\vee t$, we have that $G_t(T,x) >0$. 
 \end{assumption}
 
 \begin{remark}
Assumption~\ref{StAss0} guarantees that the random times $\tau^x$ are not stopping times with respect to  $\bbf$ and is standard in the context of many intensity models, see e.g. \ci{BR02}.  The filtration $\bbf$ contains background information about the likelihood of death events, but not their actual occurrence;  in the context of Cox processes this typically means that $\bbf$ is generated by the covariates, but not by the individual death events. Within our context of mortality models, we will actually demonstrate in Section~\ref{sec:lln}  that also predictions of survival ratios based on all available information lead to the $\bbf$-forward survival process, see Theorem~\ref{LLN}.

Note that Assumption~\ref{StAss0}  implies that a maximal age does not exist at which all individuals are necessarily dead. While unrealistic, this limitation is mitigated, if conditional survival properties are very low for old individuals, and does therefore not pose any serious restriction.
\end{remark}

\begin{definition}\label{def-hazard}
Let $-x \in \bbr$ be arbitrary.
The family of \emph{$\bbf$-forward hazard processes} is defined by $$ [0,T] \ni t \mapsto \Gamma_t(T,x) = - \ln G_t(T,x), \quad T \geq -x\vee 0 .$$ The family of \emph{$\bbf$-conditional hazard processes} is defined by $t\mapsto \Gamma_t(t,x)$, $t\geq - x\vee 0$.
\end{definition}

\begin{definition}\label{def-rates}
Let $-x \in \bbr$ be arbitrary.
\begin{enumerate}
\item If $t\mapsto \Gamma_t(t,x)$, $ t \geq - x\vee 0 $,  is absolutely continuous with respect to Lebesgue measure, i.e., 
\begin{equation}\label{spot} 
\Gamma_t(t,x) = \Gamma_0(0,x)  + \int_{ - x\vee 0} ^t \gamma_s(x) ds
\end{equation}
for a $\bbf$-optional process $\gamma (x) = (\gamma_t (x))_{t\geq -x\vee 0}$, then $\gamma (x)$ is called the \emph{$\bbf$-spot mortality rate} (or sometimes \emph{intensity}). In (\ref{spot}), we  set $\Gamma_0(0,x) := 0$ for $x \leq 0$.

\item If $T\mapsto \Gamma_t(T,x)$, $ T \geq - x\vee t$ for $t \in \bbr_+$, is absolutely continuous with respect to Lebesgue measure, i.e., 
\begin{equation}\label{m-rate}
\Gamma_t(T,x) = \Gamma_t(-x \vee t,x) + \int_{ - x\vee t}^T \mu_t(s,x) ds
\end{equation}
for $\bbf$-optional processes $\mu(T,x) = (\mu_t(T,x))_{t \in [0,T]}$, $T\geq -x\vee 0$, then the processes $\mu(T,x)$ are called the \emph{$\bbf$-forward mortality rates}. In (\ref{m-rate}), we set $\Gamma_t(-x,x) := 0$ for $t \leq -x$.

\item \label{def-j} If for $t \in \bbr_+$ the $\bbf$-forward mortality rates $\mu(T,x)$ exist for $ T \geq -x \vee t$, and if $h \mapsto \mu_t(T+h, x-h)$ is absolutely continuous with respect to Lebesgue measure, i.e.,
$$  \mu_t(T+h, x-h) -  \mu_t(T,x) = -\int _0^{h} j_t(T+u, x-u) du $$
for $\bbf$-optional processes $j(T,x) = (j_t(T,x))_{t \in [0,T]}$, $T \geq -x  \vee 0$, 
then the processes $j(T,x)$ are called \emph{$\bbf$-forward mortality improvements}.
\end{enumerate}
\end{definition}

 We will investigate the dynamics of the forward mortality rates and improvements in Section~\ref{IDF}. The martingale condition in Remark~\ref{rem-martingale} imposes restrictions on their evolution that will be studied in detail.

 \begin{remark}\mbox{} \begin{enumerate}
 \item The conventions $\Gamma_0(0,x) = 0$ for $x \leq 0$ and $\Gamma_t(-x,x) = 0$ for $t \leq -x$ in \eqref{spot} and \eqref{m-rate}, respectively, encode that all individuals are supposed to be alive before their date of birth $-x$ with probability $1$.  By Definition~\ref{def-hazard} the forward hazard process needs thus to be $0$ in these cases.

 \item Although we assume only the existence of a $\bbf$-optional $\bbf$-spot mortality rate, one could require that the spot mortality rate is predictable. This is implied by the continuity of the process $(\Gamma_t(t,x))_t$. A given  $\bbf$-optional measurable spot mortality rate $\gamma$ is not necessarily predictable, but defines a measure $ \gamma_s(x)(\omega)ds \bbp (d\omega)$  on the optional $\sigma$-algebra $\mathcal{O}$, which contains the predictable $\sigma$-algebra $\pcal$. The Radon-Nikodym-derivative of its restriction to $\pcal$ with respect to $\bbp(d\omega) dt $ is a $\bbf$-predictable spot mortality rate satisfying equation~\eqref{spot}.   
 \item The spot and forward mortality rates can be interpreted as the infinitesimal rate at which individuals die given the current information. In order to be more precise, if $T\geq -x \vee t$ and the conditions of Definition~\ref{def-rates} are satisfied, then
$$\bbp \left( T < \tau^x \leq  T+\epsilon \,|\, \fil_t \right)   = \mu_t(T,x)  \epsilon +   o(\epsilon)\quad  \mbox{ as } \epsilon \to 0 . $$ 

 \item The $\bbf$-forward mortality improvements $j(T,x)$ quantify the infinitesimal improvements of $\bbf$-forward mortality rates across cohorts. Intuitively, the increment
 $$  \mu_t(T+du, x-du) -  \mu_t(T,x) = - j_t(T,x) du   $$
 describes the changes of forward mortalities for two cohorts with identical age at two time horizons; cohort $x-du$ is of age $T+x$ at time $T+du$, while cohort $x$ is of the same age at time $T$. The forward mortality improvement $j_t(T,x)$ is positive, if the forward mortality rate decreases, and negative, if the forward mortality rate increases.
 
 Stochastic dynamics of the $\bbf$-forward mortality improvements $j$ can capture future random cohort effects. A framework for modeling their stochastic evolution is discussed in Section~\ref{IDF} below.
\end{enumerate} 
 \end{remark}

\section{A Conditional Law of Large Numbers}\label{sec:lln}

The principle of pooling is key to insurance mathematics. It states that in a very large population idiosyncratic risk per insured almost vanishes. This is traditionally used as the basis for the computation of risk premia of insurance products. In this section we apply the idea of pooling within the context of our model and prove a conditional law of large numbers showing that the $\bbf$-survival and $\bbf$-forward survival processes capture the systematic risk associated with stochastic mortality.

We fix a date of birth $-x\in \bbr$ and consider a large homogeneous family of individuals born at this date. Our goal is to compute the fraction of individuals alive at a future date $t$ as well as the best time-$t$-prediction of the fraction of individuals alive at time $T>t$; the best prediction will be based on the full information available at $t$ that does not only include the background information, but also the information about all death events. We first state an assumption that formalizes that we are considering a homogeneous population.
\begin{assumption}\label{StAss}
Let $\tau^x$ and $\hat\tau ^x$ be the times of death of two arbitrary individuals born at time $-x\in\bbr$. Then $\bbp ( \tau^x > t  \,|\, \fil_t ) = \bbp ( \hat\tau^x > t  \,|\, \fil_t )$ almost surely for all $t\in \bbr_+$. 
 \end{assumption}
Limit theorems for a population require a large pool of individuals. We thus consider a countably infinite family of individuals born at date $-x$.
\begin{remark} 
Mathematically, the existence of an infinite, but countable collection of random times on a sufficiently rich probability space $(\Omega, \gcal, \bbp)$ with given hazard process satisfying Assumption~\ref{StAss} can easily be shown using the canonical construction of random times, a notion well-known in the literature on reduced-form credit risk models.  Suppose that $\Gamma_t(t, x)$ is an increasing $\bbf$-adapted process. Letting $\varepsilon^n$ be a sequence of independent unit exponentially distributed random variables independent of $\fil_\infty := \bigvee_{t\in \bbr_+} \fil_t$, the random times $$  \tau^{x,n}:= \inf\{t: \Gamma_t(t,x)>\varepsilon^n   \}  $$
  are conditionally independent given $\fil_\infty$, each with hazard process $\Gamma_t(t,x).$
\end{remark}

\begin{definition}\label{CondI}
Letting $x\in\bbr$, a family of associated death times $(\tau^{x,n})_{n\in\bbn}$ is called \emph{$\bbf$-doubly stochastic conditionally independent  ($\bbf$-DSCI)}, if 
\begin{enumerate}
\item\label{CondI1} the sequence  $(\tau^{x,n})_{n\in\bbn}$ is doubly stochastic, i.e., $$\bbp(\tau^{x,n} >t \,|\, \fil_{t}) =\bbp(\tau^{x,n} >t \,|\, \fil_\infty) \mbox{ for all } t\in \bbr_+, n\in \bbn; $$
\item\label{CondI2} the sequence  $(\tau^{x,n})_{n\in\bbn}$ is $\fil_\infty$-conditionally independent, i.e.,  for any finite $J\subseteq \bbn$, $t_j\in\bbr_+$, $j\in J$: $$\bbp(\tau^{x,j} >t_j \mbox{ for all } j \in J \,|\, \fil_\infty) = \prod_{j\in J}  \bbp(\tau^{x,j} >t_j \,|\, \fil_\infty).$$
\end{enumerate}
This definition can canonically be extended to families of death times of countably many different cohorts.
\end{definition}

Property (i) of Definition~\ref{CondI} states that the probability of death of an individual up to time $t$ depends only on the background information up to time $t$, but not on background information arriving later. In the special context of Cox process intensities this means, for example, that the intensity at time $t$ is a function of the paths of the factor processes only up to time $t$.

Property (ii) formalizes that death times are independent given the background information.  Note that this excludes contagion effects in the sense of local or global (mean-field) interaction. This assumption is relatively innocent in the context of mortality modeling, as long as large-scale epidemic outbreaks are neglected.

\begin{theorem}\label{LLN}
Letting $x\in\bbr$, we denote by $(N^{ n} (x))_{n\in \bbn}$ the survival indicators associated to a family of individuals born at date $-x$ with $\bbf$-DSCI death times. Then for all $t \geq -x \vee 0$:
$$ \lim_{N\to \infty} \frac 1 N \sum_{n=1}^N N^{n}(x)_t  =  G_t(t,x)  \quad \bbp\mbox{--almost surely}.$$
Setting $\gcal_t: = \sigma\{  N^n(x)_s: s\leq t, n\in \bbn   \} \vee \fil_t$, $t\in\bbr_+$, $\bbg = (\gcal_t)_{t \in \bbr_+}$ is the full information filtration. Then for all $t\in\bbr_+$, $T\geq -x\vee t$:
\begin{eqnarray}\label{eq:lln} 
\lim_{N\to \infty} \frac 1 N \sum_{n=1}^N \bbe[N^{n}(x)_T \,|\, \gcal_t]  =  G_t(T,x)  \quad \bbp\mbox{--almost surely}.
\end{eqnarray}
\end{theorem}

\begin{proof}
See Appendix~\ref{app:lln}.
\end{proof}

Theorem~\ref{LLN} is a law of large numbers. The $\bbf$-conditional survival process describes mortality on the aggregate level of large populations. The quantity $G_t(t,x)$ equals the fraction of individuals born at time $-x$ that survive until $t$. The $\bbf$-forward survival process describes the best predictions of these fractions. As defined in the Theorem, $\bbg = (\gcal_t)_{t \in \bbr_+}$ models the full information that is available at time $t$. It includes both the background information as well as information about the occurrence of all death events up to time $t$. The quantity $G_t(T,x)$ is the time $t$ best estimate of the fraction of individuals born at time $-x$ that survive until $T$. We emphasize that $G_t(T,x)$ is $\fil_t$-measurable, thus depends only on the background information. The rational behind this property is pooling: idiosyncratic risks are not relevant anymore on the aggregate level.  

Theorem~\ref{LLN}  does also allow us to extend the martingale property of  $(G_t(T,x))_t$  from the background filtration $\bbf$ to the full filtration $\bbg$. The bounded convergence theorem allows us to interchange the order of the expectation and the limit in \eqref{eq:lln}. This shows that $G_t(T,x) = \bbe [G_T(T,x) | \gcal_t]$ proving the $\bbg$-martingale property. 

Finally observe that the stochastic process  $(G_t(t,x))_{t \geq -x\vee 0}$ is $\bbp$--almost surely decreasing  by Theorem~\ref{LLN}. This is also apparent from Definition~\ref{CondI}(i), since $G_t(t,x) = \bbp (\tau^{x,1}>t \,|\, \fil_t) = \bbe (N^1(x)_t \,|\, \fil_\infty) $ and $(N^1(x)_t)_{t\in\bbr_+}$ is decreasing.

\begin{remark}
The laws of large numbers refer to large homogeneous populations with cohort $x$ fixed. However, in applications, liabilities and mortality derivatives are usually linked to inhomogeneous pools  that include, for example, various cohorts. Suppose that we are interested in market-consistent valuation and that $\bbp$ is interpreted as a reference measure. In this case, the relevant cohorts need to be weighted in a suitable way. These weights can easily be encoded by Borel measures on the real line, as explained in Section~4 of \ci{BiffisMil}.
\end{remark}

\section{Compensators}\label{sec:comp}

Although quite convenient, the approach described  in Section~\ref{DEP} is not always employed in the literature on doubly-stochastic point processes. An alternative procedure for describing the probabilistic properties of the death times considers the $\bbg$-compensators. In order to facilitate a translation between both concepts, we briefly summarize some basic relationships. 

For simplicity, we restrict attention to individuals born after time $0$, i.e. we assume that $-x\geq 0$. Let $(\tau^{x,n})_{n\in\bbn}$ again be a $\bbf$-DSCI family of death times with $-x\in \bbr_+$. The corresponding survival indicators are denoted by  $(N^{ n} (x))_{n\in \bbn}$. As defined in Theorem~\ref{LLN}, $\bbg$ denotes the full information filtration.  We start by defining the notion of a compensator. Its existence follows from the Doob-Meyer-decomposition theorem. The compensator is unique, up to indistinguishability.
\begin{definition}
A $\bbg$-predictable, right-continuous, increasing process $A^n(x)$ is a $\bbg$-compensator of $\tau^{x,n}$, if  $A^n(x)_{t} = 0$, $ t\leq -x$, and 
$$
1- N^n(x)_t -A^n(x)_t,  \quad  t \geq -x,
$$
is a $\bbg$-martingale.
\end{definition}

The relationship between the $\bbg$-compensators of all death times and the $\bbf$-survival process is now described by the following proposition.

\begin{proposition}\label{MHP} Assume that $(\tau^{x,n})_{n\in\bbn}$ is a $\bbf$-DSCI family of death times with $-x\in \bbr_+$.
Then there exists a $\bbf$-predictable, right-continuous, increasing process $\Lambda(x)$   with the following property:

 $$\Lambda(x)^{\tau^{x,n}} = A^n(x) \mbox{ for all }n\in \bbn,$$ 
 where $\Lambda(x)^{\tau^{x,n}}$ signifies the process $\Lambda(x)$ stopped at $\tau^{x,n}$. The process $\Lambda(x)$  is unique, up to  indistinguishability.
 
 In other words, $\Lambda(x)$ is
a $\bbf$-predictable, right-continuous, increasing process such that  $$  1- N^n(x)_t -\Lambda (x)_{t \wedge \tau^{x,n}}, \quad t\geq -x, $$
is a $\bbg$-martingale for all $n\in \bbn$, and $\Lambda(x)_{t} =0$, $t\leq -x$. This means that $\Lambda(x)$ is a $(\bbf, \bbg)$-martingale hazard process of $\tau^{x,n}$ for any $n\in \bbn$ according to Definition 6.1.1 in \ci{BR02}.

Letting $(\tilde F_t(x))_t$ be the unique $\bbf$-predictable, increasing process with $\tilde F_{t} (x) = 0$, $t \leq -x$, such that $$(1- G_t(t,x) - \tilde F_t(x))_{t\geq -x}$$ is a $\bbf$-martingale, then $\Lambda(x)$ is given by 
\begin{equation}\label{lambda-form}  \Lambda(x)_t = \int_{(-x,t]} \frac {1}{G_{t-}(t-,x)} d \tilde F_t(x) . \end{equation}
If $t\mapsto G_t(t,x)$ is almost surely continuous, then $\tilde F_t(x) = 1- G_t(t,x)$, thus 
\begin{align*}
\Lambda(x)_t = - \int_{(-x,t]} \frac {1}{G_{t}(t,x)} d G_t(t,x). 
\end{align*}
\end{proposition}

 \begin{proof}
 See Appendix~\ref{app:comp}.
 \end{proof}

The following proposition states that under regularity conditions the joint $(\bbf, \bbg)$-martingale hazard process coincides with the $\bbf$-conditional hazard process.
\begin{proposition}\label{prop:GL}
Let $-x\in\bbr_+$.
If the $\bbf$-conditional hazard process $(\Gamma_t(t,x))_{t\geq -x }$ is continuous, then $\Gamma_t(t,x) = \Lambda(x)_t$ almost surely for $t\geq -x$.
\end{proposition}
 \begin{proof}
 See Appendix~\ref{app:comp}.
 \end{proof}

\section{Infinite-Dimensional Formulation}\label{IDF}

In this section, we provide a model  for the stochastic evolution of the $\bbf$-forward mortality rates $\mu$ and the $\bbf$-forward mortality improvements $j$ on the time interval $\bbr_+ = [0,\infty)$. We will define these quantities for each fixed $(T,x)$ as an It\^{o} process driven by a (possibly infinite dimensional) Wiener process and a compensated Poisson random measure. In a second step, we shall transform these systems of It\^{o} processes to obtain a single infinite-dimensional stochastic process having values in an adequate function space. In the framework of bond markets, this idea originates from \ci{Musiela}. As we pointed out in Remark \ref{rem-martingale}, the $\bbf$-survival processes $G(T,x)$ must necessarily be martingales. This property leads to consistency conditions which we will describe for all considered quantities.

We shall now present the general stochastic framework. We begin with the driving Wiener process $W$ taking values in some separable Hilbert space $U$ with covariance operator $Q\in L(U)$. For details we refer to Chapter 4 in \ci{Da_Prato}. The standard space of stochastic integrands with respect to $W$ consists of stochastic processes with values in the space of Hilbert-Schmidt
operators $L_2^0(H) := L_2(U_0,H)$ from $U_0 := Q^{1/2}(U)$ into some separable Hilbert space $H$. The space $H$ is the state-space of the model; adequate choices will be discussed later. The (possibly infinite-dimensional) integrands  can be decomposed into one-dimensional components on the basis of the spectral decomposition of $Q$. To be precise, we denote by $(\lambda_k)_{k \in \mathbb{N}} \subset (0,\infty)$  the sequence of non-zero eigenvalues of $Q$ and by $(e_k)_{k \in \mathbb{N}}$ the corresponding orthonormal basis of eigenvectors. Then the one-dimensional components of a $L_2^0(H)$-valued integrand $\Phi$ are given by 
\begin{align}\label{isom-isom}
\Phi^k := \Phi ( \sqrt{\lambda_k} e_k ), \quad  k \in \mathbb{N}.
\end{align}

As a second stochastic driver of the evolution we introduce a  time-homogeneous Poisson random measure  $\mathfrak{p}$ that allows to include jumps. For details we refer to \ci[Def. II.1.20]{Jacod-Shiryaev}. The mark space  $(E,\mathcal{E})$ of the Poisson random measure $\mathfrak{p}$ is a measurable space. For technical reasons we assume that $(E,\mathcal{E})$  is a Blackwell
space (see \ci{Dellacherie} or \ci{Getoor}). Blackwell spaces  include Polish spaces as a special case.  The compensator of $\mathfrak{p}$ is of the form $dt \otimes \nu(d\xi)$
for  a $\sigma$-finite measure $\nu$ on $(E,\mathcal{E})$. For further reference,  we set $\lcal_{\nu}^2(H) := \lcal^2(E,\ecal,\nu;H)$.

We present our main results in the following Sections~\ref{sec-HJM-rates}--\ref{sec-Musiela-impr}. For convenience of the reader, technical assumptions and proofs are deferred to Appendix~\ref{sec-proofs}. An illustrative example is provided in Section~\ref{sec-example}.

\subsection{Consistent HJM type dynamics of the forward mortality rates}\label{sec-HJM-rates}

In this section, we will specify the dynamics of the $\bbf$-forward mortality rates $\mu_t(T,x)$. We begin by choosing a suitable parameter domain $\Theta$ for the values of $(t,T,x)$. The variable $t$ represents the running time, the parameter $-x$ signifies the date of birth of an individual, and $T$ denotes a future date. Natural parameter restrictions are thus provided by  $0 \leq t \leq T$ , $T \geq -x$. The relevant domain $\Theta \subset \bbr_+^2 \times \bbr$ is hence given by
\begin{align}\label{def-domain-Theta}
\Theta := \{ (t,T,x) \in \bbr_+^2 \times \bbr : (T,x) \in \Xi \text{ and } t \in [0,T] \},
\end{align}
with
\begin{align}\label{def-domain-Xi}
\Xi := \{ (T,x) \in \bbr_+ \times \bbr : -x \leq T \}.
\end{align}
Next, we specify the stochastic dynamics of the forward mortality rates. Suppose that $\mu_0 : \Xi \rightarrow \mathbb{R}$ is an initial surface of $\bbf$-forward mortality rates. We suppose that the $\bbf$-forward mortality rates $\mu(T,x)$, $(T,x) \in \Xi$, follow an It\^{o} process:
\begin{equation}\label{mortality-forces}
\begin{aligned}
\mu_t(T,x) &:= \mu_0(T,x) + \int_{0}^t \alpha_s(T,x) ds + \int_{0}^t \sigma_s(T,x) dW_s
\\ &\quad + \int_{0}^t \int_E \delta_s(T,x,\xi) (\mathfrak{p}(ds,d\xi) - \nu(d\xi)ds), \quad t \in [0,T].
\end{aligned}
\end{equation}
Here, $\alpha : \Omega \times \Theta \rightarrow \mathbb{R}$, $\sigma : \Omega \times \Theta \rightarrow L_2^0(\mathbb{R})$ and $\delta : \Omega \times \Theta \times E \rightarrow \bbr$ are stochastic processes  that satisfy the technical Assumption \ref{ass-alpha} (see below) that guarantees  that all stochastic integrals in equation \eqref{mortality-forces} exist.  

For fixed $x \in \bbr$ we introduce the $\bbf$-spot mortality rates $\gamma(x)$ as
\begin{align}\label{def-gamma}
\gamma_t(x) := \mu_t(t,x) \mathbbm{1}_{\{ t \geq -x \}}, \quad t \geq 0.
\end{align}
The spot mortality rates follow an It\^{o} process as characterized in the following proposition.
\begin{proposition}
We suppose that:
\begin{itemize}
\item For all $(t,x) \in \Xi$ and $\xi \in E$ the mappings $T \mapsto \mu_0(T,x)$, $T \mapsto \alpha_t(T,x)$, $T \mapsto \sigma_t(T,x)$ and $T \mapsto \delta_t(T,x,\xi)$ are differentiable on their domains.
\item Assumption \ref{ass-alpha} holds as stated as well as for the derivatives $\partial_T \mu_0$, $\partial_T \alpha$, $\partial_T \sigma$ and $\partial_T \gamma$ in place of  $ \mu_0$, $ \alpha$, $ \sigma$ and $ \gamma$.
\end{itemize} 
Then, for each $x \in \mathbb{R}_+$ the process  $\gamma(x)$ of ~$\bbf$-spot mortality rates is an It\^{o} process of the form
\begin{align*}
\gamma_t(x) &= \gamma_0(x) + \int_0^t \zeta_u(x) du + \int_0^t \sigma_u(u,x) dW_u 
\\ &\quad + \int_{0}^t \int_E \delta_u(u,x,\xi) (\mathfrak{p}(ds,d\xi) - \nu(d\xi)ds), \quad t \geq 0,
\end{align*}
where the process $\zeta : \Omega \times \Xi \rightarrow \bbr$ is given by
\begin{align*}
\zeta_u(x) &= \alpha_u(u,x) + \partial_u \mu_0(u,x) + \int_0^u \partial_u \alpha_s(u,x)ds + \int_0^u \partial_u \sigma_s(u,x)dW_s
\\ &\quad + \int_{0}^u \int_E \partial_u \delta_s(u,x,\xi) (\mathfrak{p}(ds,d\xi) - \nu(d\xi)ds), \quad u \geq 0.
\end{align*}
\end{proposition}
\begin{proof}
The proof is analogous to that of \ci[Prop. 6.1]{fillbook}, and therefore omitted.
\end{proof}

As pointed out in Remark~\ref{rem-martingale}, the forward survival process satisfies a martingale condition which is key to the following analysis. The condition implies a crucial  consistency condition for the dynamics:
\begin{definition}\label{def-consistent-G}
The $\bbf$-forward mortality rates $\mu$ are called \emph{consistent}, if for all $(T,x) \in \Xi$ the $\bbf$-survival process $G(T,x)$  is a $\bbf$-martingale.
\end{definition}

The following theorem provides a precise criterion for the consistency  of the mortality rates. Technical assumptions are again deferred to  Appendix~\ref{sec-proofs}. In particular, we require an exponential integrability condition on the L\'{e}vy measure $\nu$ which is stated in Assumption \ref{ass-Levy-measure}.

\begin{theorem}\label{thm-drift-HJM-mu}
Suppose that Assumptions \ref{ass-alpha} and \ref{ass-Levy-measure} are satisfied. Then the $\bbf$-forward mortality rates $\mu$ are consistent if and only if
\begin{equation}\label{drift-cond}
\begin{aligned}
\alpha_t(T,x) &= \sum_{k \in \mathbb{N}} \sigma_t^k(T,x) \int_{-x \vee t}^T \sigma_t^k(u,x) du 
\\ &\quad - \int_E \delta_t(T,x,\xi) \bigg[ \exp \bigg( - \int_{-x \vee t}^T \delta_t(u,x,\xi)du \bigg) - 1 \bigg] \nu(d\xi)
\\ &\quad \text{for all $(T,x) \in \Xi$ with $T \geq t$}, \quad \text{$d \mathbb{P} \otimes dt$--almost surely on $\Omega \times \bbr_+$.}
\end{aligned}
\end{equation}
\end{theorem}

\begin{proof}
See Appendix \ref{sec-proof-1}.
\end{proof}

\begin{remark}
Theorem \ref{thm-drift-HJM-mu} provides a criterion for the consistency of the $\bbf$-forward mortality rates $\mu$ showing that the drift $\alpha$ is completely determined by the volatilities $\sigma$ and $\gamma$. This resembles the HJM drift condition for bond markets, cf. \ci{HJM} for the Wiener driven case, and \ci{BKR} for the general situation with an additional Poisson random measure. Note that our mortality model is specified with respect to  the reference measure $\bbp$ while the HJM drift condition for interest rate models holds only with respect to  an equivalent martingale measure. In the current paper, $\bbp$ may also play the role of the statistical measure.
\end{remark}

\subsection{Consistent HJM type dynamics of the forward mortality improvements}\label{sec-improvements-HJM}

In this section we specify the dynamics of the forward mortality improvements and derive consistency conditions.  Suppose that $j_0 : \Xi \rightarrow \mathbb{R}$ is an initial surface of $\bbf$-forward mortality improvements.  We assume that the $\bbf$-forward mortality improvements $j(T,x)$, $(T,x) \in \Xi$, follow an It\^{o} process:
\begin{equation}\label{j-dynamics}
\begin{aligned}
j_t(T,x) &:= j_0(T,x) + \int_{0}^t a_s(T,x) ds + \int_{0}^t b_s(T,x) dW_s
\\ &\quad + \int_{0}^t \int_E c_s(T,x,\xi) (\mathfrak{p}(ds,d\xi) - \nu(d\xi)ds), \quad t \in [0,T].
\end{aligned}
\end{equation}
Here,  $a : \Omega \times \Theta \rightarrow \mathbb{R}$, $b : \Omega \times \Theta \rightarrow L_2^0(\mathbb{R})$, and $c : \Omega \times \Theta \times E \rightarrow \bbr$ are stochastic processes satisfying the technical Assumption \ref{ass-a} (see below). Assumption \ref{ass-a} ensures that the stochastic integrals in Definition \eqref{j-dynamics} exist. 

Letting $\gamma_0 : \bbr_+ \rightarrow \bbr$ be an initial curve of $\bbf$-spot mortality rates, we extend this curve to an initial surface $\mu_0 : \Xi \rightarrow \bbr$ of $\bbf$-forward mortality rates by setting
\begin{align}\label{def-mu-0-extend}
\mu_0(T,x) := \gamma_0(T+x) - \int_0^T j_0(u,T+x-u) du.
\end{align}
Now, using the dynamics of the forward mortality improvements as the starting point, we redefine the stochastic processes $\alpha : \Omega \times \Theta \rightarrow \mathbb{R}$, $\sigma : \Omega \times \Theta \rightarrow L_2^0(\mathbb{R})$ and $\delta : \Omega \times \Theta \times E \rightarrow \bbr$ as
\begin{equation}\label{def-alpha-a}
\begin{aligned}
&\alpha_t(T,x) := - \int_t^T a_t(u,T+x-u)du, \quad \sigma_t(T,x) := - \int_t^T b_t(u,T+x-u)du \quad \text{and}
\\ &\delta_t(T,x,\xi) := - \int_t^T c_t(u,T+x-u,\xi)du.
\end{aligned}
\end{equation}
For each $(T,x) \in \Xi$ we redefine the $\bbf$-forward mortality rates $\mu(T,x)$ by (\ref{mortality-forces}). Again, we impose an exponential integrability condition on the L\'{e}vy measure $\nu$ formalized by Assumption \ref{ass-Levy-measure-c}.

\begin{theorem}\label{thm-drift-j-1}
Suppose that Assumption \ref{ass-a} is satisfied. Then the following statements are true:
\begin{enumerate}
\item For all $x \in \bbr_+$ the $\bbf$-spot mortality rates $\gamma(x)$ follow the dynamics
\begin{equation}\label{spot-shifted-simple}
\begin{aligned}
\gamma_t(x) &= \mu_0(t,x) + \int_0^t \alpha_s(t,x) ds + \int_0^t \sigma_s(t,x) dW_s
\\ &\quad + \int_0^t \int_E \delta_s(t,x,\xi) (\mathfrak{p}(ds,d\xi) - \nu(d\xi)ds), \quad t \geq -x \vee 0.
\end{aligned}
\end{equation}

\item For all $(T,x) \in \Xi$ we have
\begin{align}\label{identity-mu-gamma-j}
\mu_t(T,x) = \gamma_t(T+x-t) - \int_t^T j_t(u,T+x-u)du, \quad t \in [0,T].
\end{align}

\item If, in addition, Assumption \ref{ass-Levy-measure-c} is satisfied, and the drift $a$ is given by
\begin{equation}\label{a-defined}
\begin{aligned}
a_t(T,x) &= -\sum_{k \in \bbn} \bigg( \int_t^T b_t^k(u,T+x-u)du \bigg) \bigg( \int_{-x \vee t}^T b_t^k(u,x) du \bigg)
\\ &\quad - \sum_{k \in \bbn} b_t^k(T,x) \int_{-x \vee t}^T \int_t^u b_t^k(v, u+x-v) dv du
\\ &\quad - \int_E \bigg( \int_t^T c_t(u,T+x-u,\xi) du  \bigg) \bigg( \int_{-x \vee t}^T c_t(u,x,\xi)du \bigg)
\\ &\qquad\qquad \times \exp \bigg( \int_{-x \vee t}^T \int_t^u c_t(v,u+x-v,\xi) dv du \bigg) \nu(d\xi)
\\ &\quad - \int_E c_t(T,x,\xi) \bigg[  \exp \bigg( \int_{-x \vee t}^T \int_t^u c_t(v,u+x-v,\xi) dv du \bigg) - 1 \bigg] \nu(d\xi),                                                                                              
\end{aligned}
\end{equation}
then the $\bbf$-forward mortality rates $\mu$ are consistent. 
\end{enumerate}
\end{theorem}

\begin{proof}
See Appendix \ref{sec-proof-2}.
\end{proof}

\begin{remark}\label{rem-procedure}
The previous results entail the following relations between the initial surfaces and the volatilities of consistent dynamic evolutions of forward mortality rates and improvements:
\begin{itemize}
\item For a given initial surface $j_0$ of $\bbf$-forward mortality improvements and an initial curve $\gamma_0$ of $\bbf$-spot mortality rates, we can compute the initial surface $\mu_0$ of $\bbf$-forward mortality rates by (\ref{def-mu-0-extend}).

\item Conversely, for a given initial surface $\mu_0$ of $\bbf$-forward mortality rates, we can compute the initial surface $j_0$ of $\bbf$-forward mortality improvements as
\begin{align*}
j_0(T,x) := -(\partial_T - \partial_x) \mu_0(T,x).
\end{align*}

\item For given volatilities $a,b,c$ in (\ref{j-dynamics}) we can compute the volatilities $\alpha,\sigma,\delta$ in (\ref{mortality-forces}) by (\ref{def-alpha-a}).

\item For given volatilities $\alpha,\sigma,\delta$ in (\ref{mortality-forces}) with $\alpha_t(t,x) = \sigma_t(t,x) = \delta_t(t,x,\xi) = 0$ we can compute the volatilities $a,b,c$ in (\ref{j-dynamics}) by
\begin{align*}
&a_t(T,x) := -(\partial_T - \partial_x) \alpha_t(T,x), \quad b_t(T,x) := -(\partial_T - \partial_x) \sigma_t(T,x) \quad \text{and}
\\ &c_t(T,x,\xi) := -(\partial_T - \partial_x) \delta_t(T,x,\xi).
\end{align*}

\item Note that, for consistency, the drift term $a$ is given by (\ref{a-defined}), and the drift term $\alpha$ is given by (\ref{drift-cond}).
\end{itemize}
\end{remark}

\subsection{Consistent Musiela type dynamics of the forward mortality rates}\label{sec-Musiela-mu}

In this section, we shall transform the parameter domain $\Theta$ in order to get a unified framework. The $\bbf$-forward mortality rates will be described by one infinite-dimensional stochastic process with values in a Hilbert space $H$ consisting of functions $h : \Xi \rightarrow \bbr$. 

\noindent For this procedure, we introduce the mapping
\begin{align*}
\phi : \Theta \rightarrow \mathbb{R}_+ \times \Xi, \quad \phi(t,T,x) = (t,T-t,x+t),
\end{align*}
which is bijective with inverse
\begin{align*}
\phi^{-1} : \mathbb{R}_+ \times \Xi \rightarrow \Theta, \quad \phi^{-1}(t,s,y) = (t,s+t,y-t).
\end{align*}

\begin{definition}\label{def-space}
We call a separable Hilbert space $(H, \| \cdot \|)$ a \emph{forward mortality space}, if the following conditions are satisfied:
\begin{enumerate}
\item $H$ consists of continuous functions $h : \Xi \rightarrow \bbr$.

\item For each $(s,y) \in \Xi$ the point evaluation $\ell_{(s,y)} : H \rightarrow \bbr$, $\ell_{(s,y)}(h) := h(s,y)$ is a continuous linear functional.

\item For each bounded Borel set $B \subset \Xi$ there is a constant $C > 0$ such that
\begin{align}\label{point-eval-uniform}
\| \ell_{(s,y)} \| \leq C \quad \text{for all $(s,y) \in B$.}
\end{align}
\item The shift semigroup $(S_t)_{t \geq 0}$ given by 
\begin{align}\label{def-semigroup}
S_t h(s,y) := h(s+t,y-t), \quad (s,y) \in \Xi
\end{align}
is a $C_0$-semigroup on $H$.
\end{enumerate}
\end{definition}

The domain $\Xi$ of functions in a forward mortality space is shown in Figure \ref{figure-domain}. The variable $s$ represents the length of the time horizon for which survival probabilities are computed;  $y$ denotes the current age of individuals of a given cohort. The sum $s+y$ is the age of individuals of cohort $y$  at the end of the time horizon $s$. Note that the current age $y$ is allowed to be negative, labeling individuals of future generations for which forecasts are made. For each $(s,y) \in \Xi$ we necessarily have $s+y \geq 0$ by the choice of the domain $\Xi$; i.e. at the end of the prediction time horizon, individuals for which predictions are made will already have been born.

Figure \ref{figure-domain} also illustrates the action of the shift semigroup, where the vector $(s,y) \in \Xi$ is mapped to $(s+t,y-t) \in \Xi$. For positive $t$ this mapping shifts quantities to younger generations with the same age.

\begin{figure}
\begin{center}
  \includegraphics[height=60mm,width=40mm,clip]{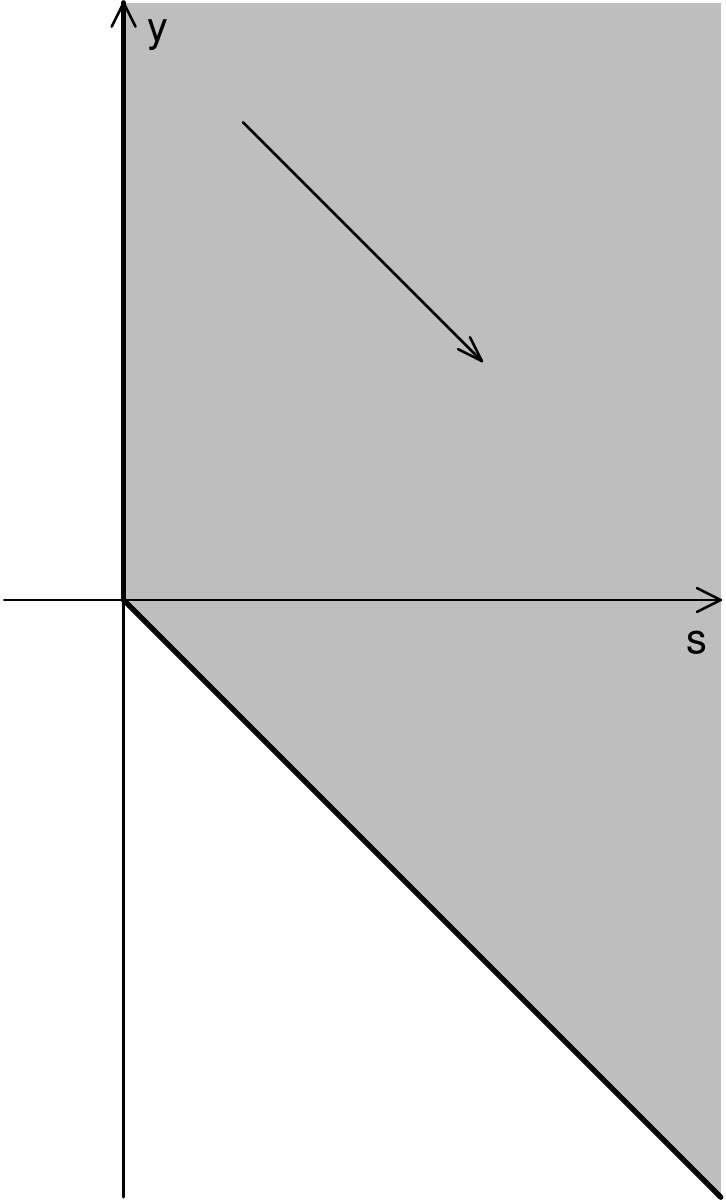}
\caption{The domain $\Xi$ and the action of the shift semigroup $(S_t)_{t \geq 0}$.}
\label{figure-domain}
\end{center}
\end{figure}

An example of a forward mortality space can be constructed as follows:

\begin{example}\label{ex-space}
Let $w_1,w_2 : \bbr_+ \rightarrow (0,\infty)$ and $w_3 : \bbr_+^2 \rightarrow (0,\infty)$ be strictly positive, continuous weight functions. A possible specific example is provided by the following parametric choice: 
\begin{equation}\label{weight-functions}
\begin{aligned}
w_1(s) &= e^{-\beta s}, \quad w_2(z) = e^{-\beta z} \quad \text{and} 
\\ w_3(s,z) &= e^{-\beta(s+z)} \quad \text{for some constant $\beta > 0$.}
\end{aligned}
\end{equation}
Let $H$ be the linear space consisting of all functions $h : \Xi \rightarrow \bbr$ satisfying the following conditions:
\begin{itemize}
\item For all $s \in \bbr_+$ the mappings $z \mapsto h(s,z-s)$, $z \mapsto \partial_s h(s,z-s)$ are absolutely continuous, and for all $z \in \bbr$ the mappings $s \mapsto h(s,z-s)$, $s \mapsto \partial_z h(s,z-s)$ are absolutely continuous (and hence, almost everywhere differentiable).

\item We have $\partial_{sz}h(s,z-s) = \partial_{zs}h(s,z-s)$ for almost all $(s,z) \in \bbr_+^2$.

\item We have
\begin{equation}\label{def-norm-H}
\begin{aligned}
\| h \| &:= \bigg( |h(0,0)|^2 + \int_0^{\infty} | \partial_s h(s,-s) |^2 w_1(s) ds + \int_0^{\infty} | \partial_z h(0,z) |^2 w_2(z) dz 
\\ &\quad\quad + \int_0^{\infty} \int_0^{\infty} | \partial_{sz} h(s,z-s) |^2 w_3(s,z) dz ds \bigg)^{1/2} < \infty.
\end{aligned}
\end{equation}
\end{itemize}
Then $(H, \| \cdot \|)$ is a forward mortality space. The arguments are similar to those from Section 5 in \ci{fillnm}: A straightforward calculation shows that for all $h \in H$ and $s_0,z_0 \in \bbr_+$ we have
\begin{equation}\label{repr-eval}
\begin{aligned}
h(s_0,z_0-s_0) &= h(0,0) + \int_0^{s_0} \partial_s h(s,-s) ds + \int_0^{z_0} \partial_z h(0,z) dz 
\\ &\quad + \int_0^{s_0} \int_0^{z_0} \partial_{sz}h(s,z-s) dz ds.
\end{aligned}
\end{equation}
The representation (\ref{repr-eval}) shows that every $h \in H$ is a continuous function. Moreover, by (\ref{repr-eval}) and the Cauchy-Schwarz inequality, we obtain (\ref{point-eval-uniform}). Using the estimate (\ref{point-eval-uniform}), for each $h \in H$ with $\| h \| = 0$ we have
$h=0$, showing that $\| \cdot \|$ is a norm (not just a seminorm) on $H$. By Definition (\ref{def-norm-H}) of the norm $\| \cdot \|$ we have
\begin{align*}
H \cong \bbr \times L^2(\bbr_+) \times L^2(\bbr_+) \times L^2(\bbr_+^2),
\end{align*}
showing that $(H,\| \cdot \|)$ is a separable Hilbert space. Finally, similar arguments as in \ci[Section 5]{fillnm} show that $(S_t)_{t \geq 0}$ is a $C_0$-semigroup on $H$.
\end{example}

From now on, let $H$ be a forward mortality space.
Denoting by $A$ the infinitesimal generator of the shift semigroup $(S_t)_{t \geq 0}$, for each $h \in \mathcal{D}(A)$ we obtain
\begin{align*}
Ah(s,y) &= \lim_{t \rightarrow 0} \frac{S_t h(s,y) - h(s,y)}{t} = \lim_{t \rightarrow 0} \frac{h((s,y) + t(1,-1)) - h(s,y)}{t}
\\ &= Dh(s,y)(1,-1) = ( \partial_s - \partial_y ) h(s,y), \quad (s,y) \in \Xi.
\end{align*}
Therefore, we have $A = \partial_s - \partial_y$ and
\begin{align*}
\mathcal{D} ( \partial_s - \partial_y ) \subset \{ h \in H:  ( \partial_s - \partial_y ) h \mbox{ exists with } ( \partial_s - \partial_y ) h \in H \}.
\end{align*}
Let $\mu_0 \in H$ be an initial surface of forward mortality rates. We define the $\bbf$-forward mortality rates $\bar{\mu}$ as the $H$-valued process
\begin{equation}\label{var-const}
\begin{aligned}
\bar{\mu}_t &:= S_t \mu_0 + \int_0^t S_{t-s} \bar{\alpha}_s ds + \int_0^t S_{t-s} \bar{\sigma}_s d W_s
\\ &\quad + \int_0^t \int_E S_{t-s} \bar{\delta}_s(\xi) ( \mathfrak{p}(ds,d\xi) - \nu(d\xi)ds ), \quad t \geq 0.
\end{aligned}
\end{equation}
Here, $\bar{\alpha} : \Omega \times \bbr_+ \rightarrow H$, $\bar{\sigma} : \Omega \times \bbr_+ \rightarrow L_2^0(H)$ and $\bar{\delta} : \Omega \times \bbr_+ \times E \rightarrow H$ are stochastic processes satisfying Assumption \ref{ass-alpha-bar} (see below) which ensures that all stochastic integrals in the variation of constants formula (\ref{var-const}) exist.

\begin{remark}\label{rem-positivity}
As we have seen in Section \ref{DEP}, forward mortality rates should be positive processes. In the context of the current paper, we do, however, not impose a general positivity assumption on the  $\bbf$-forward mortality rates $\bar{\mu}$ defined in (\ref{var-const}) resp. on the $\bbf$-forward mortality rates $\mu$ defined in (\ref{mortality-forces}). Positivity depends on the choice of the specific model. Even if positivity is violated,  in practice,  one can ensure by an adequate choice of the parameters that the mortality rates  become negative only with low probability. Such a model may then be regarded as a good approximation of reality; this modeling philosophy conceptually parallels the approach of the Vasi\u{c}ek model and its Hull-White extension in the context of interest rates. 

In order to rigorously investigate positivity in a general framework, one would, as in \ci{Positivity} for classical HJM models, define the closed, convex cone
\begin{align*}
P := \bigcap_{(s,y) \in \Xi} \{ h \in H : h(s,y) \geq 0 \}
\end{align*}
of nonnegative mortality surfaces and derive appropriate conditions for stochastic invariance of $P$ for (\ref{var-const}). 
\end{remark}
\begin{definition}
The $\bbf$-forward mortality rates $\bar{\mu}$ are called \emph{consistent}, if the transformed $\bbf$-forward mortality rates $\mu := \bar{\mu} \circ \phi$ are consistent in the sense of Definition \ref{def-consistent-G}.
\end{definition}
In the following theorem we impose again an exponential integrability condition on the L\'{e}vy measure $\nu$.
\begin{theorem}\label{thm-drift-HJM}
Suppose that Assumptions \ref{ass-alpha-bar} and \ref{ass-Levy-measure-bar} are fulfilled. Then the $\bbf$-forward mortality rates $\bar{\mu}$ are consistent if and only if
\begin{equation}\label{drift-cond-Musiela}
\begin{aligned}
\bar{\alpha}_t(s,y) &= \sum_{k \in \mathbb{N}} \bar{\sigma}_t^k(s,y) \int_{-y \vee 0}^{s} \bar{\sigma}_t^k(u,y) du
\\ &\quad - \int_E \bar{\delta}_t(s,y,\xi) \bigg[ \exp \bigg( -\int_{-y \vee 0}^{s} \bar{\delta}_t(u,y,\xi) du \bigg) - 1 \bigg] \nu(d\xi)
\\ &\quad \text{for all $(s,y) \in \Xi$}, \quad \text{$d\bbp \otimes dt$--almost surely on $\Omega \times \bbr_+$.} 
\end{aligned}
\end{equation}
\end{theorem}

\begin{proof}
See Appendix \ref{sec-proof-3}.
\end{proof}

\begin{remark}
In the spirit of \ci{Da_Prato}, the process $\bar{\mu}$ of $\bbf$-forward mortality rates in (\ref{var-const}) is a mild solution to the stochastic partial differential equation (SPDE) 
\begin{align*}
\left\{
\begin{array}{rcl}
d \bar{\mu}_t & = & \big( (\partial_s - \partial_y) \bar{\mu}_t + \bar{\alpha}_t \big) dt + \bar{\sigma}_t dW_t + \bar{\delta}_t(\xi) ( \mathfrak{p}(dt,d\xi) - \nu(d\xi)dt ) \medskip
\\ \bar{\mu}_0 & = & \mu_0.
\end{array}
\right.
\end{align*}
Condition (\ref{drift-cond-Musiela}) is necessary and sufficient for consistency of the $\bbf$-forward mortality rates. It resembles the HJM drift condition for bond markets with Musiela parametrization.
\end{remark}

\subsection{Consistent Musiela type dynamics of the forward mortality improvements}\label{sec-Musiela-impr}

In this section, we derive a unified framework for the $\bbf$-forward mortality improvements. Stochastic processes will take values in an appropriate function space. The implied $\bbf$-forward mortality rates will be consistent under suitable conditions.

Let $H$ be a forward mortality space (see Definition \ref{def-space}) and   $j_0 \in H$  an initial surface of forward mortality improvements. We define the $\bbf$-forward mortality improvements $\bar{j}$ as the $H$-valued process
\begin{equation}\label{j-var-const}
\begin{aligned}
\bar{j}_t &:= S_t j_0 + \int_0^t S_{t-s} \bar{a}_s ds + \int_0^t S_{t-s} \bar{b}_s d W_s
\\ &\quad + \int_0^t \int_E S_{t-s} \bar{c}_s(\xi) ( \mathfrak{p}(ds,d\xi) - \nu(d\xi)ds ), \quad t \geq 0.
\end{aligned}
\end{equation}
Here,  $\bar{a} : \Omega \times \bbr_+ \rightarrow H$, $\bar{b} : \Omega \times \bbr_+ \rightarrow L_2^0(H)$ and $\bar{c} : \Omega \times \bbr_+ \times E \rightarrow H$ are stochastic processes satisfying the technical Assumption \ref{ass-a-bar} (see below). Assumption \ref{ass-a-bar}  ensures that the stochastic integrals in the variation of constants formula (\ref{j-var-const})  exist. 

Letting $\gamma_0 : \bbr_+ \rightarrow \bbr$ be an initial curve of $\bbf$-spot mortality rates, we extend this curve to an initial surface $\mu_0 : \Xi \rightarrow \bbr$ of $\bbf$-forward mortality rates by setting
\begin{align}\label{def-mu-0-extend-Musiela}
\mu_0(s,y) := \gamma_0(s+y) - \int_0^s j_0(u,s+y-u) du.
\end{align}
Furthermore, we redefine the stochastic processes $\bar{\alpha} : \Omega \times \Xi \rightarrow \mathbb{R}$, $\bar{\sigma} : \Omega \times \Xi \rightarrow L_2^0(\mathbb{R})$ and $\bar{\delta} : \Omega \times \Xi \times E \rightarrow \bbr$ as
\begin{equation}\label{def-alpha-a-bar}
\begin{aligned}
&\bar{\alpha}_t(s,y) := - \int_0^s \bar{a}_t(u,s+y-u)du, \quad
\bar{\sigma}_t(s,y) := - \int_0^s \bar{b}_t(u,s+y-u)du \quad \text{and}
\\ &\bar{\delta}_t(s,y,\xi) := - \int_0^s \bar{c}_t(u,s+y-u,\xi)du.
\end{aligned}
\end{equation}
Suppose that $\mu_0 \in H$, and that $\bar{\alpha}$, $\bar{\sigma}$ and $\bar{\delta}$ are $H$-valued processes such that Assumption \ref{ass-alpha-bar} is fulfilled. The $H$-valued process $\bar{\mu}$ is redefined by the variation of constants formula (\ref{var-const}), and for $y \geq 0$ we define the $\bbf$-spot mortality rates $\bar{\gamma}(y)$ as
\begin{align}\label{def-gamma-bar}
\bar{\gamma}_t(y) := \bar{\mu}_t(0,y), \quad t \geq 0.
\end{align}
Again, we impose an exponential integrability condition on the L\'{e}vy measure $\nu$.

\begin{theorem}\label{thm-main-Musiela}
Suppose that Assumption \ref{ass-a-bar} is satisfied. Then the following statements are true:
\begin{enumerate}
\item For all $y \in \bbr_+$ the $\bbf$-spot mortality rates $\bar{\gamma}(y)$ have the dynamics
\begin{equation}\label{spot-shifted-bar}
\begin{aligned}
\bar{\gamma}_t(y) &= S_t \mu_0(0,y) + \int_0^t S_{t-s} \bar{\alpha}_s(0,y) ds + \int_0^t S_{t-s} \bar{\sigma}_s(0,y) dW_s
\\ &\quad + \int_0^t \int_E S_{t-s} \bar{\delta}_s(0,y,\xi) (\mathfrak{p}(ds,d\xi) - \nu(d\xi)ds), \quad t \geq 0.
\end{aligned}
\end{equation}

\item For all $(s,y) \in \Xi$ we have
\begin{align}\label{identity-mu-gamma-j-bar}
\bar{\mu}_t(s,y) = \bar{\gamma}_t(s+y) - \int_0^s \bar{j}_t(u,s+y-u)du, \quad t \geq 0.
\end{align}

\item If, in addition, Assumption \ref{ass-Levy-measure-c-bar} is satisfied, and the drift $\bar a$ is given by
\begin{equation}\label{a-bar-defined}
\begin{aligned}
\bar{a}_t(s,y) &= -\sum_{k \in \bbn} \bigg( \int_0^s \bar{b}_t^k(u,s+y-u)du \bigg) \bigg( \int_{-y \vee 0}^s \bar{b}_t^k(u,y) du \bigg)
\\ &\quad - \sum_{k \in \bbn} \bar{b}_t^k(s,y) \int_{-y \vee 0}^s \int_0^u \bar{b}_t^k(v, u+y-v) dv du
\\ &\quad - \int_E \bigg( \int_0^s \bar{c}_t(u,s+y-u,\xi) du  \bigg) \bigg( \int_{-y \vee 0}^s \bar{c}_t(u,y,\xi)du \bigg)
\\ &\qquad\qquad \times \exp \bigg( \int_{-y \vee 0}^s \int_0^u \bar{c}_t(v,u+y-v,\xi) dv du \bigg) \nu(d\xi)
\\ &\quad - \int_E \bar{c}_t(s,y,\xi) \bigg[  \exp \bigg( \int_{-y \vee 0}^s \int_0^u \bar{c}_t(v,u+y-v,\xi) dv du \bigg) - 1 \bigg] \nu(d\xi),                                                                                              
\end{aligned}
\end{equation}
then the $\bbf$-forward mortality rates $\bar{\mu}$ are consistent. 
\end{enumerate}
\end{theorem}

\begin{proof}
See Appendix \ref{sec-proof-4}.
\end{proof}

\begin{remark}\label{rem-procedure-Musiela}
In the spirit of \ci{Da_Prato}, the process $\bar{j}$ of $\bbf$-forward mortality improvements in (\ref{j-var-const}) is a mild solution to the stochastic partial differential equation (SPDE) 
\begin{align*}
\left\{
\begin{array}{rcl}
d \bar{j}_t & = & \big( (\partial_s - \partial_y) \bar{j}_t + \bar{a}_t \big) dt + \bar{b}_t dW_t + \bar{c}_t(\xi) ( \mathfrak{p}(dt,d\xi) - \nu(d\xi)dt ) \medskip
\\ \bar{j}_0 & = & j_0.
\end{array}
\right.
\end{align*}
Similar to Remark \ref{rem-procedure}, we can derive the following relations between the initial surfaces and the volatilities of consistent dynamic evolutions of forward mortality rates and improvements:
\begin{itemize}
\item For a given initial surface $j_0$ of $\bbf$-forward mortality improvements and an initial curve $\gamma_0$ of $\bbf$-spot mortality rates, we can compute the initial surface $\mu_0$ of $\bbf$-forward mortality rates by (\ref{def-mu-0-extend-Musiela}).

\item Conversely, for a given initial surface $\mu_0$ of $\bbf$-forward mortality rates, we can compute the initial surface $j_0$ of $\bbf$-forward mortality improvements as
\begin{align*}
j_0(s,y) := -(\partial_s - \partial_y) \mu_0(s,y).
\end{align*}

\item For given volatilities $\bar{a},\bar{b},\bar{c}$ in (\ref{j-var-const}) we can compute the volatilities $\bar{\alpha},\bar{\sigma},\bar{\delta}$ in (\ref{var-const}) by (\ref{def-alpha-a-bar}).

\item For given volatilities $\bar{\alpha},\bar{\sigma},\bar{\delta}$ in (\ref{var-const}) with $\bar{\alpha}_t(0,y) = \bar{\sigma}_t(0,y) = \bar{\delta}_t(0,y,\xi) = 0$ we can compute the volatilities $\bar{a},\bar{b},\bar{c}$ in (\ref{j-var-const}) by
\begin{align*}
&\bar{a}_t(s,y) := -(\partial_s - \partial_y) \bar{\alpha}_t(s,y), \quad \bar{b}_t(s,y) := -(\partial_s - \partial_y) \bar{\sigma}_t(s,y) \quad \text{and}
\\ &\bar{c}_t(s,y,\xi) := -(\partial_s - \partial_y) \bar{\delta}_t(s,y,\xi).
\end{align*}

\item Note that, for consistency, the drift term $\bar{a}$ is given by (\ref{a-bar-defined}), and the drift term $\bar{\alpha}$ is given by (\ref{drift-cond-Musiela}).
\end{itemize}
\end{remark}

\section{Example: A L\'{e}vy process driven Gompertz-Makeham model}\label{sec-example}

In order to illustrate our previous results, we present a L\'{e}vy process driven version of the Gompertz-Makeham model, and compute consistent dynamics of this model. In Section \ref{sec-Levy-1} we consider the general situation, where the $\bbf$-forward mortality rates and the $\bbf$-forward mortality improvements are driven by a L\'{e}vy process; in Section \ref{sec-Levy-2} we focus on the particular case of the Gompertz-Makeham model. 

\subsection{L\'{e}vy process driven mortality models}\label{sec-Levy-1}

Let $X$ be a real-valued L\'{e}vy process with Gaussian part $C \geq 0$ and L\'{e}vy measure $\nu$. We assume that there exist constants $M,\epsilon > 0$ such that
\begin{align}\label{cond-cumulant}
\int_{\{ |\xi| > 1 \}} e^{z\xi} \nu(d\xi) < \infty \quad \text{for all $z \in [-(1+\epsilon)M,(1+\epsilon)M]$.}
\end{align}
Then the cumulant generating function $\Psi(z) := \ln \bbe [e^{z X_1}]$ exists on $[-(1+\epsilon)M, (1+\epsilon)M]$ and is of class $C^{\infty}$ on the interior $(-(1+\epsilon)M, (1+\epsilon)M)$ with representations
\begin{align*}
\Psi(z) &= Bz + \frac{C}{2}z^2 + \int_{\bbr} ( e^{z \xi} - 1 - z \xi ) \nu(d\xi),
\\ \Psi'(z) &= B + Cz + \int_{\bbr} \xi ( e^{z \xi} - 1 ) \nu(d\xi),
\\ \Psi''(z) &= C + \int_{\bbr} \xi^2 e^{z \xi} \nu(d\xi),
\end{align*}
where $B \in \bbr$ denotes the drift of $X$.
We shall directly switch to Musiela type dynamics. Let $H$ be a forward mortality space, see Definition \ref{def-space}. Suppose that the $\bbf$-forward mortality rates $\bar{\mu}$ and the $\bbf$-forward mortality improvements $\bar{j}$ are given by the variation of constants formulas
\begin{align}\label{Levy-mu}
\bar{\mu}_t &= S_t \mu_0 + \int_0^t S_{t-s} \hat{\alpha}_s ds + \int_0^t S_{t-s} \hat{\sigma}_s dX_s, \quad t \geq 0,
\\ \label{Levy-j} \bar{j}_t &= S_t j_0 + \int_0^t S_{t-s} \hat{a}_s ds + \int_0^t S_{t-s} \hat{b}_s dX_s, \quad t \geq 0,
\end{align}
with initial surfaces $\mu_0, j_0 \in H$ and   appropriate $H$-valued processes $\hat{\alpha}, \hat{\sigma}$ and $\hat{a}, \hat{b}$. Note that these are particular cases of the variation of constant formulas (\ref{var-const}), (\ref{j-var-const}):  the state space $U$ of the Wiener process $W$ and the mark space $E$ of the Poisson random measure $\mathfrak{p}$ are $U = E = \bbr$; the mapping $\bar{\sigma}$ in (\ref{var-const}) is provided by $\sqrt{C} \hat{\sigma}$, and $\bar{\delta}(\xi)$ in (\ref{var-const}) is given by $\xi \hat{\sigma}$; the mapping $\bar{b}$ in (\ref{j-var-const}) by $\sqrt{C} \hat{b}$, and $\bar{c}(\xi)$ in (\ref{j-var-const}) by $\xi \hat{b}$.

In this case, the drift condition (\ref{drift-cond-Musiela}) in Theorem \ref{thm-drift-HJM} for the $\bbf$-forward mortality rates ensuring consistency becomes
\begin{align}\label{drift-mu-Levy}
\hat{\alpha}_t(s,y) = - \hat{\sigma}_t(s,y) \Psi' \bigg( -\int_{-y \vee 0}^s \hat{\sigma}_t(u,y)du \bigg),
\end{align}
and the drift condition (\ref{a-bar-defined}) in Theorem \ref{thm-main-Musiela} for the $\bbf$-forward mortality improvements translates to
\begin{equation}\label{drift-j-Levy}
\begin{aligned}
\hat{a}_t(s,y) &= -\bigg( \int_0^s \hat{b}_t(u,s+y-u)du \bigg) \bigg( \int_{-y \vee 0}^s \hat{b}_t(u,y)du \bigg) 
\\ &\quad \cdot \Psi'' \bigg( \int_{-y \vee 0}^s \int_0^u \hat{b}_t(v,u+y-v) dv du \bigg)
\\ &\quad - \hat{b}_t(s,y) \Psi' \bigg( \int_{-y \vee 0}^s \int_0^u \hat{b}_t(v,u+y-v) dv du \bigg).
\end{aligned}
\end{equation}

\subsection{The Gompertz-Makeham model}\label{sec-Levy-2}

As an illustrating example of a stochastic mortality model, we describe in the current section a L\'{e}vy process driven Gompertz-Makeham model. Letting $\theta_1 > 1$ and $\theta_2, \theta_3, \theta_4, \theta_5 > 0$ be real  constants, the initial surface $j_0 : \Xi \rightarrow \bbr$ of $\bbf$-forward mortality improvements, the initial surface $\mu_0 : \Xi \rightarrow \bbr$ of $\bbf$-forward mortality rates and the initial curve $\gamma_0 : \bbr_+ \rightarrow \bbr$ of $\bbf$-spot mortality are provided by
\begin{align*}
j_0(s,y) &= \theta_2 e^{-\theta_2 s} ( \theta_3 e^{\theta_4(s+y)} + \theta_5),
\\ \mu_0(s,y) &= (\theta_1 + e^{-\theta_2 s}) (\theta_3 e^{\theta_4(s+y)} + \theta_5),
\\ \gamma_0(y) &= (\theta_1 + 1) (\theta_3 e^{\theta_4 y} + \theta_5).
\end{align*}
These initial surfaces satisfy the relationships described in Remark \ref{rem-procedure-Musiela}. For every $z \in \bbr_+$ we observe that
\begin{align*}
\lim_{s \rightarrow \infty} \mu_0(s,z-s) = \theta_1 (\theta_3 e^{\theta_4 z} + \theta_5),
\end{align*}
i.e., if the length $s$ of the prediction time horizon tends to $\infty$ for predictions about the mortality of individuals at age $z$ at the end of the time horizon, then the initial forward mortality rates are again described by a classical Gompertz-Makeham model.

\begin{remark}
The initial surfaces $j_0$ and $\mu_0$ belong to the forward mortality space $H$ defined in Example \ref{ex-space}, with appropriate choices of weight functions (\ref{weight-functions}).
\end{remark}

Finally, we describe three examples of volatility structures $\hat{a},\hat{b}$ and $\hat{\alpha},\hat{\sigma}$, which are computed according to Remark \ref{rem-procedure-Musiela} and drift conditions (\ref{drift-mu-Levy}), (\ref{drift-j-Levy}).

\begin{example}\label{ex-b-1}
If the volatility $\hat{b}$ of the $\bbf$-forward mortality improvements is constant and equal to $1$, then we compute
\begin{align*}
\hat{a}(s,y) &= -s(s + y \mathbbm{1}_{\{ y < 0 \}})  \Psi'' \bigg( \frac{s^2}{2} - \frac{y^2}{2} \mathbbm{1}_{\{ y < 0 \}} \bigg) -  \Psi' \bigg( \frac{s^2}{2} - \frac{y^2}{2} \mathbbm{1}_{\{ y < 0 \}} \bigg),
\\ \hat{b}(s,y) &= 1,
\\ \hat{\alpha}(s,y) &= s \Psi' \bigg( \frac{s^2}{2} - \frac{y^2}{2} \mathbbm{1}_{\{ y < 0 \}} \bigg),
\\ \hat{\sigma}(s,y) &= -s.
\end{align*}
In particular, the volatility of the forward mortality rates is proportional to the length of the prediction time horizon.
\end{example}

\begin{example}\label{ex-b-2}
In this example, the volatility $\hat{b}$ of the $\bbf$-forward mortality improvements equals the age $s+y$ of individuals at the end of the time period. In this case, we obtain
\begin{align*}
\hat{a}(s,y) &= -s(s+y) \bigg( \frac{s^2}{2} + sy + \frac{y^2}{2} \mathbbm{1}_{\{ y < 0 \}} \bigg) \Psi'' \bigg( \frac{3s^2y + 2s^3}{6} - \frac{y^3}{6} \mathbbm{1}_{\{ y < 0 \}} \bigg)
\\ &\quad -(s+y) \Psi' \bigg( \frac{3s^2y + 2s^3}{6} - \frac{y^3}{6} \mathbbm{1}_{\{ y < 0 \}} \bigg),
\\ \hat{b}(s,y) &= s+y,
\\ \hat{\alpha}(s,y) &= s(s+y) \Psi' \bigg( \frac{3s^2y + 2s^3}{6} - \frac{y^3}{6} \mathbbm{1}_{\{ y < 0 \}} \bigg),
\\ \hat{\sigma}(s,y) &= -s(s+y).
\end{align*}
In particular, the volatility of the forward mortality rates is proportional to the length of the prediction time horizon times the age of the individuals at the end of the time horizon.
\end{example}

\begin{example}\label{ex-b-3}
Finally, the volatility $\hat{b}$ of the $\bbf$-forward mortality improvements is assumed to equal the age $s+y$ at the end of the time horizon multiplied by $1 - e^{-s}$, a factor which models increasing uncertainty for longer prediction horizons. Introducing the auxiliary functions $g,h : \Xi \rightarrow \bbr$ as
\begin{align*}
g(s,y) &= \frac{(2 s e^s + 2) y + s^2 e^s + 2s + 2}{2} e^{-s} - (y+1)\mathbbm{1}_{\{ y < 0 \}} - \frac{2 e^y - y^2}{2} \mathbbm{1}_{\{ y \geq 0 \}},
\\ h(s,y) &= \frac{((3s^2 - 6s) e^s - 6)y + (2s^3 - 3s^2)e^s - 6s - 6}{6} e^{-s} + (y+1) \mathbbm{1}_{\{ y \geq 0 \}}
\\ &\quad + \frac{6 e^y - y^3 - 3y^2}{6} \mathbbm{1}_{\{ y < 0 \}},
\end{align*}
we obtain the drift and volatility terms
\begin{align*}
\hat{a}(s,y) &= -(s+y)( 1 + s - e^{-s}) g(s,y) \Psi''(h(s,y)) - (s+y)(1 - e^{-s}) \Psi'(h(s,y)),
\\ \hat{b}(s,y) &= (s+y)(1 - e^{-s}),
\\ \hat{\alpha}(s,y) &= (s+y)( 1 + s - e^{-s}) \Psi' (h(s,y)),
\\ \hat{\sigma}(s,y) &= -(s+y)( e^{-s} + s - 1).
\end{align*}
\end{example}

The $\bbf$-forward mortality rates $\mu$ described in Examples \ref{ex-b-1}--\ref{ex-b-3} may generally become negative with non zero probability; they should, thus, be interpreted as approximations of real mortality rates, see Remark \ref{rem-positivity}.

\begin{remark}
For the volatilities $\hat{b}$ from Examples \ref{ex-b-1}--\ref{ex-b-3} the double integrals appearing in (\ref{drift-j-Levy}), which are evaluated by $\Psi'$ and $\Psi''$, take values in $\bbr_+$. Therefore, in the context of these examples one should assume -- besides condition (\ref{cond-cumulant}) -- that the cumulant generating function $\Psi$ exists even on $\bbr_+$. This hypothesis is, for example, satisfied if the L\'{e}vy process is a jump diffusion $X = W + N$ that can be described as the sum of a Wiener process $W$ and a Poisson process $N$. In this particular case, the cumulant generating function equals
\begin{align*}
\Psi(z) = \frac{z^2}{2} + e^z - 1.
\end{align*}
For appropriate choices of weight functions $w_1$, $w_2$ and $w_3$, the drift and volatility terms in Examples \ref{ex-b-1}--\ref{ex-b-3} thus belong to the forward mortality space $H$ defined in Example \ref{ex-space}.
\end{remark}

\subsection*{Acknowledgements}
We acknowledge very useful comments and suggestions by the editor, the associate editor and two anonymous referees. We thank Jan Baldeaux, Anna-Maria Hamm, Eckhard Platen, Cord-Roland Rinke, Thomas Salfeld, Max Stollmann and Sven Wiesinger for helpful remarks and discussions.

\normalsize

\begin{appendix}

\appendix\normalsize

\section{Appendix to Section~\ref{sec:lln}: A Conditional Law of Large Numbers}\label{app:lln}

In this appendix, we provide the proof of Theorem \ref{LLN}.

\begin{proof}[Proof of Theorem~\ref{LLN}]
The random variables $N^n(x)_t$, $n\in\bbn$, are conditionally independent given $\fil_\infty$ with identical $\fil_\infty$-conditional Bernoulli-distributions by Assumption~\ref{StAss} and Definition~\ref{CondI}. By a conditional strong law of large numbers, see e.g. Theorem 7 in \ci{rao}, we have
$$\frac 1 N \sum_{n=1}^N N^n(x)_t \;\longrightarrow \;\bbe[N^1(x)_t \,|\, \fil_\infty] \;= \;\bbp(\tau^{x,1} >t \,|\, \fil_t) \;=\; G_t(t,x)\quad\quad \bbp\mbox{--almost surely.}$$
By a conditional version of Lebesgue's dominated convergence theorem, see e.g. Theorem 23.8 in \ci{japr}, we obtain
\begin{eqnarray*}&& \lim_{N\to\infty} \frac 1 N \sum_{n=1}^N \bbe[N^n(x)_T\,|\,\gcal_t]  =   \bbe\left[ \lim_{N\to\infty} \frac 1 N \sum_{n=1}^N N^n(x)_T \,\bigg|\,\gcal_t\right]  = \bbe[G_T(T,x)\,|\,\gcal_t] \\ && \stackrel{(*)}{=} \bbe[G_T(T,x)\,|\,\fcal_t] \; = \;G_t(T,x) .\end{eqnarray*}
Equality $(*)$ follows by Proposition 6.6 in \ci{kall}, if $\fil_T \underset{\fil_t}{\independent} \sigma\{N^n(x)_s: s\leq t, n\in \bbn\} .$ Since $\fil_T \subset \fil_{\infty}$, this follows from Lemma~\ref{CILLN}. 
\end{proof}

\begin{lemma}\label{CILLN}
Consider the setting of Theorem~\ref{LLN}. Then for $t\in\bbr_+$: $\fil_\infty \underset{\fil_t}{\independent} \sigma\{N^n(x)_s: s\leq t, n\in \bbn\} .$
\end{lemma}

\begin{proof}
If $B\in\sigma(N^n(x)_s)$ for some $s \leq t$, then $\mathbbm{1}_B = f(N^n(x)_s)$ with $f:\{0,1\} \to \{0,1\}$. Letting $A\in\fil_{\infty}$, $K\in\bbn$, $s_k\leq t$, $f_k:\{0,1\}\to\{0,1\}$, $k=1,\dots, K$, we have
 \begin{align*}
 &\bbp\left(   A \cap  \bigcap_{k=1}^K \{f_k(N^k(x)_{s_k}) = 1 \} \,\bigg|\, \fil_t   \right) = \bbe \left[  \bbe \left[  \mathbbm{1}_A \cdot \prod_{k=1}^K f_k(N^k(x)_{s_k})     \,\bigg|\, \fil_\infty \right] \, \bigg| \, \fil_t \right] \\
 &= \bbe \left[   \mathbbm{1}_A \cdot \bbe \left[  \prod_{k=1}^K f_k(N^k(x)_{s_k})     \,\bigg|\, \fil_\infty \right]  \,\bigg|\, \fil_t \right] = \bbe \left[   \mathbbm{1}_A \cdot \prod_{k=1}^K \bbe \left[   f_k(N^k(x)_{s_k})     \,|\, \fil_t \right]  \bigg| \fil_t \right]\\
 &= \bbe \left[   \mathbbm{1}_A \,|\, \fil_t \right]  \cdot \prod_{k=1}^K \bbe \left[   f_k(N^k(x)_{s_k})   \, | \, \fil_t \right],\\ 
 \end{align*}
and hence
\begin{align*}
&\bbp\left(   A \cap  \bigcap_{k=1}^K \{f_k(N^k(x)_{s_k}) = 1 \} \,\bigg|\, \fil_t   \right) = \bbp \left(   A \,|\, \fil_t \right)  \cdot \bbe \left[ \prod_{k=1}^K \bbe \left[   f_k(N^k(x)_{s_k})     \,|\, \fil_t \right] \, \bigg| \,  \fil_t \right]  \\
 &= \bbp \left(   A\,|\, \fil_t \right) \cdot \bbe \left[ \prod_{k=1}^K \bbe \left[   f_k(N^k(x)_{s_k})  \,|\, \fil_\infty \right] \, \bigg| \,  \fil_t \right]   \\
  &=  \bbp \left(   A\,|\, \fil_t \right) \cdot \bbe \left[ \bbp \left(  \bigcap_{k=1}^K \{f_k(N^k(x)_{s_k}) = 1 \}    \,\bigg|\, \fil_\infty \right) \, \bigg| \, \fil_t \right]\\ & = \bbp \left(   A\,|\, \fil_t \right) \cdot \bbp \left(  \bigcap_{k=1}^K \{f_k(N^k(x)_{s_k}) = 1 \}   \, \bigg| \,  \fil_t \right).
 \end{align*}
The class of subsets $$\ccal := \left\{ \bigcap_{k=1}^K \{f_k(N^k(x)_{s_k}) = 1  
\} : s_k\leq t,\; f_k:\{0,1\}\to\{0,1\}, \; k\in\{1,\dots, K\}, \;   K\in\bbn     \right\}$$ is a $\pi$-system (i.e. closed under finite intersections) that generates the $\sigma$-algebra $ \sigma\{N^n(x)_s: s\leq t, n\in \bbn\} $. The lemma now follows from a conditional analogue of Lemma~3.6 in \ci{kall}.
\end{proof}

\section{Appendix to Section~\ref{sec:comp}: Compensators}\label{app:comp}

In this appendix, we provide the proofs of Propositions~\ref{MHP} and \ref{prop:GL}.

\begin{proof}[Proof of Proposition~\ref{MHP}]
Define the filtration $\bbg^n= (\gcal^n_t)_{t\in \bbr_+}$ by setting $\gcal^n_t = \fil_t \vee\sigma(N^n(x)_s: s\leq t)$, $t\in\bbr_+$. The $\bbg^n$-compensator $B^n(x)$ of $\tau^{x,n}$ coincides with the $\bbg$-compensator $A^n(x)$. This can be seen as follows: the $\bbg^n$-compensator $B^n(x)$ of $\tau^{x,n}$ is a $\bbg$-predictable, right-continuous, increasing process. If $1 - N^n(x) - B^n(x)$ was a $\bbg$-martingale, then $B^n(x)$ would be equal to the $\bbg$-compensator $A^n(x)$. The martingale property follows from Lemma~\ref{CIComp}.

Formula \eqref{lambda-form} thus defines a $(\bbf, \bbg)$-martingale hazard process of $\tau^{x,n}$ by Proposition 6.1.2 in \ci{BR02}. Since formula \eqref{lambda-form} does not depend on $n$, the existence of a process $\Lambda(x)$ with the desired properties is proven. It remains to show uniqueness. Since $\lim_{N\to\infty} \frac 1 N \sum_{n=1}^N N^n(x)_t = G_t(t,x) > 0$ $\bbp$--almost surely by Theorem~\ref{LLN} and Assumption~\ref{StAss0}, we know that the increasing sequence of $\bbg$-stopping times 
$\tilde \tau^{x,n} := \max\{  \tau^{x,i}: i=1, \dots, n  \},  $ $n\in\bbn$, diverges $\bbp$--almost surely to $\infty$. Observe that $\Lambda(x)^{\tilde \tau^{x,n}} = A^i(x)$ on $\{\tilde \tau ^{x,n} = \tau^{x,i}\}$ and $\Omega =\bigcup_{i=1}^n \{ \tilde \tau ^{x,n} = \tau ^{x,i}\}$. Since the $\bbg$-compensators $A^i(x)_{i\in \bbn}$ are unique, it follows that the stopped process $\Lambda(x)^{\tilde \tau^{x,n}}$ is uniquely specified. Since $\tilde \tau^{x,n} \to \infty$ $\bbp$--almost surely for $n\to \infty$, this implies the uniqueness of $\Lambda(x)$.

If $t\mapsto G_t(t,x)$ is almost surely continuous, then it is also predictable, whence the last statement follows.
\end{proof}

For the proof of Proposition \ref{prop:GL} we prepare an auxiliary result:

\begin{lemma}\label{CIComp}
Assume that $(\tau^{x,j})_{j\in\bbn}$ is a $\bbf$-DSCI family of death times of individuals born at date $-x\in \bbr_+$. Define $\gcal_t: = \sigma\{  N^j(x)_s: s\leq t, j\in \bbn   \} \vee \fil_t$,   $\gcal^n_t = \sigma(N^n(x)_s: s\leq t) \vee \fil_t$,    $t\in\bbr_+$, $\gcal^n_\infty= \bigvee_{t\in \bbr_+} \gcal^n_t$ for some $n\in \bbn$. 

Let $Y$ be a $\gcal^n_\infty$-measurable, integrable random variable. Then for $t\in\bbr_+$:
$$  \bbe[Y\,|\,\gcal_t]  =  \bbe[Y\,|\,\gcal^n_t] \quad\quad  \bbp\mbox{--almost surely}. $$
\end{lemma}
\begin{proof}
Let $\hcal^n_t= \sigma(N^n(x)_s: s\leq t)$, $\hcal^{\neq n}_t= \sigma(N^j(x)_s: s\leq t, j\in \bbn, j\neq n)$, $\hcal_t= \sigma(N^j(x)_s: s\leq t, j\in\bbn)$ and $\hcal^n_\infty= \bigvee_{t\in \bbr} \hcal^n_t$. By Lemma~\ref{CILLN} we have $\fil_\infty \underset{\fil_t}{\independent} \hcal_t = \hcal^n_t \vee \hcal^{\neq n}_t.$ This implies by Proposition~6.8 in \ci{kall}:
\begin{equation}\label{x1}
\fil_\infty \underset{\fil_t\vee \hcal^n_t}{\independent} \hcal^{\neq n}_t.
\end{equation}
It follows from Definition~\ref{CondI}(ii) that $\hcal^n_t\vee \hcal^n_\infty = \hcal^n_\infty \underset{\fil_\infty}{\independent} \hcal_t^{\neq n}$, thus by Proposition~6.8 in \ci{kall}:
\begin{equation}\label{x2}
\hcal^n_\infty \underset{\fil_\infty\vee \hcal^n_t}{\independent} \hcal^{\neq n}_t.
\end{equation}
Equations \eqref{x1} and \eqref{x2} imply by Proposition~6.8 in \ci{kall} that $$\gcal^n_\infty = \fil_\infty\vee \hcal^n_\infty \underset{\fil_t\vee \hcal^n_t}{\independent} \hcal^{\neq n}_t.$$  Since $\fil_t\vee \hcal^n_t = \gcal^n_t$ and $\hcal^{\neq n}_t \vee \gcal^n_t = \gcal_t$, Corollary~6.7(i) in \ci{kall} shows that $ \gcal^n_\infty  \underset{\gcal^n_t}{\independent} \gcal_t .$ We obtain $\bbe[Y \,|\, \gcal_t] = \bbe[Y \,|\, \gcal_t \vee \gcal_t^n] =  \bbe[Y \,|\,  \gcal_t^n] ,$ where the last equality follows from Proposition~6.6 in \ci{kall}.
\end{proof}

 \begin{proof}[Proof of Proposition~\ref{prop:GL}]
  Lemma~\ref{CIComp} shows that $\Lambda(x)$ is a $(\bbf, \bbg^n)$-martingale hazard process of $\tau^{x,n}$ for any $n\in \bbn$, where $\bbg^n=(\gcal^n_t)_{t\in \bbr_+}$ is the filtration introduced in the proof of Proposition~\ref{MHP}. Hence, by Proposition~6.2.1(ii) in \ci{BR02} we obtain
 \begin{align*}
 \Lambda(x)_t = -\ln(1 - \tilde{F}_t(x)) = -\ln G_t(t,x) = \Gamma_t(t,x), \quad t \geq -x,
 \end{align*}
 proving the claim.
\end{proof}

\section{Appendix to Section~\ref{IDF}: Infinite-Dimensional Formulation}\label{sec-proofs}

In this appendix, we provide the proofs of Section \ref{IDF} as well as technical assumptions.

\subsection{Proof of Theorem \ref{thm-drift-HJM-mu}}\label{sec-proof-1}

In this appendix, we provide the proof of Theorem \ref{thm-drift-HJM-mu}. Motivated by  Definitions (\ref{def-domain-Theta}) and (\ref{def-domain-Xi}) of $\Theta$ and $\Xi$,
for each $x \in \bbr$ we define the sets $\Xi_x \subset \bbr_+$ and $\Theta_x \subset \bbr_+ \times \bbr$ as
\begin{align*}
\Xi_x := \{ T \in \bbr_+ : (T,x) \in \Xi \} \quad \text{and} \quad \Theta_x := \{ (t,T) \in \bbr_+^2 : (t,T,x) \in \Theta \}.
\end{align*}

\begin{assumption}\label{ass-alpha}
We suppose that the following conditions are satisfied:
\begin{enumerate}
\item $\mu_0$ is $\bcal(\Xi)$-measurable, $\alpha$ and $\sigma$ are $\mathcal{F}_{\infty} \otimes \mathcal{B}(\Theta)$-measurable, and $\delta$ is $\mathcal{F}_{\infty} \otimes \mathcal{B}(\Theta) \otimes \mathcal{E}$-measurable.

\item For all $(T,x) \in \Xi$ the processes $\alpha(T,x)$ and $\sigma(T,x)$ are optional, and $\delta(T,x)$ is predictable.

\item For each $x \in \bbr$ and every bounded Borel set $B \subset \Xi_x$ we have $\int_B |\mu_0(T,x)| dT < \infty$.

\item For each $x \in \bbr$ and every bounded Borel set $B \subset \Theta_x$ there are random variables $X^{\alpha} : \Omega \rightarrow \bbr$ and $X^{\sigma}, X^{\delta} \in \lcal^2(\Omega)$ such that
\begin{equation}\label{cond-alpha}
\begin{aligned}
&|\alpha_t(T,x)| \leq X^{\alpha}, \quad \| \sigma_t(T,x) \|_{L_2^0(\bbr)} \leq X^{\sigma} \quad \text{and}
\\ &\| \delta_t(T,x) \|_{\lcal_{\nu}^2(\bbr)} \leq X^{\delta} \quad \text{for all $(t,T) \in B$.}
\end{aligned}
\end{equation}

\item For all $t \in \bbr_+$ and $\xi \in E$ the mappings $(T,x) \mapsto \alpha_t(T,x)$, $(T,x) \mapsto \sigma_t(T,x)$ and $(T,x) \mapsto \delta_t(T,x,\xi)$ are continuous on their domains. 
\end{enumerate}
\end{assumption}

Condition (\ref{cond-alpha}) ensures that all subsequent stochastic integrals regarding $\alpha$, $\sigma$ and $\delta$ exist. In particular, for each $x \in \bbr$ and every bounded Borel set $B \subset \Theta_x$ we have
\begin{equation}\label{cond-alpha-Fubini}
\begin{aligned}
&\iint_B |\alpha_t(T,x)| dt dT < \infty,
\quad \bbe \bigg[ \iint_B \| \sigma_t(T,x) \|_{L_2^0(\bbr)}^2 dt dT \bigg] < \infty \quad \text{and}
\\ &\bbe \bigg[ \iint_B \| \delta_t(T,x) \|_{\lcal_{\nu}^2(\bbr)}^2 dt dT \bigg] < \infty.
\end{aligned}
\end{equation}
This ensures that we may apply the classical Fubini theorem and the stochastic Fubini theorems (Theorem 2.8 in \ci{Gawarecki-Mandrekar} and Theorem A.2 in \ci{SPDE}) later on.

\begin{remark}\label{rem-loc-martingale}
 Definition (\ref{def-G-official}) shows that for all $(T,x) \in \Xi$ we have $|G(T,x)| \leq 1$. Therefore, the $\bbf$-forward mortality rates $\mu$ are consistent if and only if for all $(T,x) \in \Xi$ the $\bbf$-survival process $G(T,x)$ is a local $\bbf$-martingale, see e.g. \ci[Prop. I.1.47]{Jacod-Shiryaev}.
\end{remark}

For the proof of Theorem \ref{thm-drift-HJM-mu}, it will be useful to extend the $\bbf$-forward mortality rates $\mu$ and $\bbf$-spot mortality rates $\gamma$ as follows. We define the process $\tilde{\mu} : \Omega \times \bbr_+^2 \times \bbr \rightarrow \bbr$ as
\begin{align*}
\tilde{\mu}_t(T,x) :=
\begin{cases}
\mu_t(T,x), & (t,T,x) \in \Theta,
\\ 0, & -x > T,
\\ \mu_T(T,x), & -x \leq T \text{ and } t > T,
\end{cases}
\end{align*}
and the process $\tilde{\gamma} : \Omega \times \bbr_+ \times \bbr \rightarrow \bbr$ as
\begin{align*}
\tilde{\gamma}_t(x) := \tilde{\mu}_t(t,x).
\end{align*}
A straightforward calculation shows that for all $(T,x) \in \Xi$ we have
\begin{align}\label{G-for-HJM}
G_t(T,x) = \exp ( -\Gamma_0(0,x) ) \exp \bigg( -\int_0^t \tilde{\gamma}_s(x) ds \bigg) \exp \bigg( - \int_t^T \tilde{\mu}_t(s,x) ds \bigg), \quad t \in [0,T].
\end{align}
We shall now determine the dynamics of the processes $\tilde{\mu}(T,x)$ and $\tilde{\gamma}(x)$ in (\ref{G-for-HJM}). For this purpose, we extend the initial surface $\mu_0$ and the processes $\alpha$, $\sigma$ and $\delta$ in (\ref{mortality-forces}) as follows.
We define the mapping
\begin{align*}
\tilde{\mu}_0 : \bbr_+ \times \mathbb{R} \rightarrow \mathbb{R}, \quad \tilde{\mu}_0(T,x) := \mu_0(T,x) \mathbbm{1}_{\Xi}(T,x),
\end{align*}
and the processes $\tilde{\alpha} : \Omega \times \bbr_+^2 \times \bbr \rightarrow \mathbb{R}$, $\tilde{\sigma} : \Omega \times \bbr_+^2 \times \bbr \rightarrow L_2^0(\mathbb{R})$ and $\tilde{\delta} : \Omega \times \bbr_+^2 \times \bbr \times E \rightarrow \mathbb{R}$ by
\begin{align*}
&\tilde{\alpha}_t(T,x) := \alpha_t(T,x) \mathbbm{1}_{\Theta}(t,T,x), \quad
\tilde{\sigma}_t(T,x) := \sigma_t(T,x) \mathbbm{1}_{\Theta}(t,T,x) \quad \text{and}
\\ &\tilde{\delta}_t(T,x,\xi) := \delta_t(T,x,\xi) \mathbbm{1}_{\Theta}(t,T,x).
\end{align*}
Then for each $(T,x) \in \Xi$ we have
\begin{equation}\label{mortality-forces-tilde}
\begin{aligned}
\tilde{\mu}_t(T,x) &= \tilde{\mu}_0(T,x) + \int_{0}^t \tilde{\alpha}_s(T,x) ds + \int_{0}^t \tilde{\sigma}_s(T,x) dW_s
\\ &\quad + \int_{0}^t \int_E \tilde{\delta}_s(T,x,\xi) (\mathfrak{p}(ds,d\xi) - \nu(d\xi)ds), \quad t \in [0,T],
\end{aligned}
\end{equation}
and for each $x \in \bbr$ we have
\begin{equation}\label{dyn-gamma-tilde}
\begin{aligned}
\tilde{\gamma}_t(x) &= \tilde{\mu}_0(t,x) + \int_{0}^t \tilde{\alpha}_s(t,x) ds + \int_{0}^t \tilde{\sigma}_s(t,x) dW_s
\\ &\quad + \int_{0}^t \int_E \tilde{\delta}_s(t,x,\xi) (\mathfrak{p}(ds,d\xi) - \nu(d\xi)ds), \quad t \geq 0.
\end{aligned}
\end{equation}
By virtue of (\ref{cond-alpha}),
we may define the processes $A : \Omega \times \Theta \rightarrow \mathbb{R}$, $\Sigma : \Omega \times \Theta \rightarrow L_2^0(\mathbb{R})$ and $\Delta : \Omega \times \Theta \times E \rightarrow \bbr$ as
\begin{align}\label{A-tilde}
A_t(T,x) &:= -\int_t^T \tilde{\alpha}_t(u,x)du = -\int_{-x \vee t}^T \alpha_t(u,x)du,
\\ \label{Sigma-tilde} \Sigma_t(T,x) &:= -\int_t^T \tilde{\sigma}_t(u,x)du = -\int_{-x \vee t}^T \sigma_t(u,x)du,
\\ \label{Delta-tilde} \Delta_t(T,x,\xi) &:= -\int_t^T \tilde{\delta}_t(u,x,\xi)du = -\int_{-x \vee t}^T \delta_t(u,x,\xi)du.
\end{align}
The integrals in (\ref{Sigma-tilde}) are Bochner integrals over the state space $L_2^0(\bbr)$. Using the notation (\ref{isom-isom}), for each $k \in \bbn$ we have
\begin{align*}
\Sigma_t^k(T,x) = -\int_t^T \tilde{\sigma}_t^k(u,x)du = -\int_{-x \vee t}^T \sigma_t^k(u,x)du.
\end{align*}

\begin{remark}\label{remark-a-capital}
By Assumption \ref{ass-alpha}, the mappings $(T,x) \mapsto A_t(T,x)$, $(T,x) \mapsto \Sigma_t(T,x)$ and $(T,x) \mapsto \Delta_t(T,x,\xi)$ are continuous, and for all $t \in \bbr_+$ and $x \in \bbr$ the mappings $T \mapsto A_t(T,x)$, $T \mapsto \Sigma_t(T,x)$ and $T \mapsto \Delta_t(T,x,\xi)$ are continuously differentiable on $[-x \vee t, \infty)$.
\end{remark}

\begin{proposition}\label{prop-G-dynamics-pre}
For all $(T,x) \in \Xi$ the $\bbf$-survival process $G(T,x)$ is an It\^{o} process with dynamics
\begin{equation}\label{G-dynamics-pre}
\begin{aligned}
&G_t(T,x) = G_0(T,x) + \int_{0}^t G_s(T,x) \Big( A_s(T,x) + \frac{1}{2} \|\Sigma_s(T,x) \|_{L_2^0(\bbr)}^2 \Big) ds
\\ &\quad + \int_{0}^t G_s(T,x) \Sigma_s(T,x) dW_s
\\ &\quad + \int_0^t \int_E G_{s-}(T,x) \Delta_s(T,x,\xi) (\mathfrak{p}(ds,d\xi) - \nu(d\xi)ds)
\\ &\quad + \int_{0}^t \int_E G_{s-}(T,x) \big( e^{\Delta_s(T,x,\xi)} - 1 - \Delta_s(T,x,\xi) \big) \mathfrak{p}(ds,d\xi), \quad t \in [0,T].
\end{aligned}
\end{equation}
\end{proposition}

\begin{proof}
By virtue of equation (\ref{G-for-HJM}) and the dynamics (\ref{mortality-forces-tilde}), (\ref{dyn-gamma-tilde}), we may argue as in the proof of \ci[Prop. 5.2]{BKR}. For the calculations, we may use the classical Fubini theorem and the stochastic Fubini theorems (see, e.g., Theorem 2.8 in \ci{Gawarecki-Mandrekar} and Theorem A.2 in \ci{SPDE}) by virtue of (\ref{cond-alpha-Fubini}). 
\end{proof}

In addition, we require the following assumption.
For $n \in \bbn$ we denote by $\Theta_n \subset \Theta$ the compact set $\Theta_n := \{ \theta \in \Theta : \| \theta \| \leq n \}$.

\begin{assumption}\label{ass-Levy-measure}
We suppose that for each $n \in \bbn$ there exist a measurable mapping $\rho_n : E \rightarrow \bbr_+$ and a constant $\epsilon_n > 0$ such that
\begin{align}\label{Levy-measure-1}
&\int_{\{ \rho_n \leq 1 \}} \rho(\xi)^2 \nu(d\xi) + \int_{\{ \rho_n > 1 \}} e^{(1+\epsilon_n)\rho_n(\xi)} \nu(d\xi) < \infty,
\\ \label{Levy-measure-2a} &|\delta_t(T,x,\xi)| \leq \rho_n(\xi) \quad \text{for all $(t,T,x) \in \Theta_n$ and $\xi \in E$.}
\end{align}
\end{assumption}

\begin{lemma}\label{lemma-A-loc}
For each $n \in \bbn$ there exists a measurable mapping $\pi_n : E \rightarrow \bbr_+$ such that
\begin{align}\label{pi-integrable}
&\int_E \pi_n(\xi) \nu(d\xi) < \infty,
\\ \label{e-Taylor-int} &\big| e^{\Delta_t(T,x,\xi)} - 1 - \Delta_t(T,x,\xi) \big| \leq \pi_n(\xi) \quad \text{for all $(t,T,x) \in \Theta_n$ and $\xi \in E$,}
\\ \label{e-Taylor} &\big| \delta_t(T,x,\xi) \big( e^{\Delta_t(T,x,\xi)} - 1 \big) \big| \leq \pi_n(\xi) \quad \text{for all $(t,T,x) \in \Theta_n$ and $\xi \in E$.}
\end{align}
\end{lemma}

\begin{proof}
By Assumption \ref{ass-Levy-measure} and Definition (\ref{Delta-tilde}) of $\Delta$, for each $n \in \bbn$ there exist a measurable mapping $\rho_n : E \rightarrow \bbr_+$ and a constant $\epsilon_n > 0$ such that conditions (\ref{Levy-measure-1}), (\ref{Levy-measure-2a}) are fulfilled and we have
\begin{align}\label{Levy-measure-2} 
&|\Delta_t(T,x,\xi)| \leq \rho_n(\xi) \quad \text{for all $(t,T,x) \in \Theta_n$ and $\xi \in E$.}
\end{align}
Let $n \in \bbn$ be arbitrary. We define the measurable mapping
\begin{align*}
\pi_n : E \rightarrow \bbr_+, \quad \pi_n(\xi) := \exp(1) \rho_n(\xi)^2 \mathbbm{1}_{\{ \rho_n \leq 1 \}} + \frac{2}{\epsilon_n^2} e^{(1+\epsilon_n) \rho_n(\xi)} \mathbbm{1}_{\{ \rho_n > 1 \}}.
\end{align*}
Then the integrability condition (\ref{pi-integrable}) is satisfied due to (\ref{Levy-measure-1}). Note that for each $m \in \bbn_0$ we have the estimate
\begin{align*}
\bigg| e^x - \sum_{k=0}^{m-1} \frac{x^k}{k!} \bigg| \leq \sum_{k=m}^{\infty} \frac{|x|^k}{k!} = |x|^m \sum_{k=0}^{\infty} \frac{|x|^k}{(k+m)!} \leq |x|^{m} e^{|x|}, \quad x \in \bbr.
\end{align*}
Moreover, for all $\epsilon > 0$ and $x > 1$ we have
\begin{align*}
x^2 e^x = \frac{2}{\epsilon^2} \frac{(\epsilon x)^2}{2} e^x \leq \frac{2}{\epsilon^2} e^{\epsilon x} e^x = \frac{2}{\epsilon^2} e^{(1+\epsilon) x}.
\end{align*}
Consequently, by (\ref{Levy-measure-2a}), (\ref{Levy-measure-2}) we deduce (\ref{e-Taylor-int}), (\ref{e-Taylor}).
\end{proof}

\begin{proposition}\label{prop-G-dynamics}
For all $(T,x) \in \Xi$ the $\bbf$-survival process $G(T,x)$ is an It\^{o} process with dynamics
\begin{equation}\label{G-dynamics}
\begin{aligned}
&G_t(T,x) = G_0(T,x) 
\\ &\quad + \int_{0}^t G_s(T,x) \bigg( A_s(T,x) + \frac{1}{2} \|\Sigma_s(T,x) \|_{L_2^0(\bbr)}^2 
\\ &\qquad\qquad\qquad\quad + \int_E \big( e^{\Delta_s(T,x,\xi)} - 1 - \Delta_s(T,x,\xi) \big) \nu(d\xi) \bigg) ds
\\ &\quad + \int_{0}^t G_s(T,x) \Sigma_s(T,x) dW_s
\\ &\quad + \int_{0}^t \int_E G_{s-}(T,x) \big( e^{\Delta_s(T,x,\xi)} - 1 \big) (\mathfrak{p}(ds,d\xi) - \nu(d\xi)ds), \quad t \in [0,T].
\end{aligned}
\end{equation}
\end{proposition}

\begin{proof}
Let $(T,x) \in \Xi$ be arbitrary. Then there exists $n \in \bbn$ such that $(t,T,x) \in \Theta_n$ for all $t \in [0,T]$. Since $|G(T,x)| \leq 1$, by estimate (\ref{e-Taylor-int}) from Lemma \ref{lemma-A-loc} we obtain
\begin{align*}
\int_0^t \int_E \big| G_s(T,x) \big( e^{\Delta_s(T,x,\xi)} - 1 - \Delta_s(T,x,\xi) \big) \big| \nu(d\xi) ds \leq \int_E \pi_n(\xi) \nu(d\xi), \quad t \in [0,T]
\end{align*}
showing that this process belongs to $\acal_{\rm loc}^+$, see \ci[Sec. I.3]{Jacod-Shiryaev}.
Now, applying \ci[Prop. II.1.28]{Jacod-Shiryaev}, by dynamics (\ref{G-dynamics-pre}) we obtain (\ref{G-dynamics}).
\end{proof}

Now, we are ready to provide the proof of Theorem \ref{thm-drift-HJM-mu}:

\begin{proof}[Proof of Theorem \ref{thm-drift-HJM-mu}] By Assumption \ref{StAss}, Remark \ref{rem-loc-martingale} and Proposition \ref{prop-G-dynamics}, the forward mortality rates (\ref{mortality-forces}) are consistent if and only if for each $(T,x) \in \Xi$ we have
\begin{equation}\label{equiv-drift-1}
\begin{aligned}
&A_t(T,x) + \frac{1}{2} \|\Sigma_t(T,x) \|_{L_2^0(\bbr)}^2 
+ \int_E \big( e^{\Delta_t(T,x,\xi)} - 1 - \Delta_t(T,x,\xi) \big) \nu(d\xi) = 0
\\ &\text{$d \mathbb{P} \otimes dt$--almost surely on $\Omega \times [0,T]$,} \quad \text{for all $(T,x) \in \Xi$.}
\end{aligned}
\end{equation}
By Remark \ref{remark-a-capital}, Lemma \ref{lemma-A-loc} and Lebesgue's dominated convergence theorem, the left-hand side of (\ref{equiv-drift-1}) is continuous in $(T,x)$. Thus, condition (\ref{equiv-drift-1}) is equivalent to
\begin{equation}\label{equiv-drift-2}
\begin{aligned}
&A_t(T,x) + \frac{1}{2} \|\Sigma_t(T,x) \|_{L_2^0(\bbr)}^2 
+ \int_E \big( e^{\Delta_t(T,x,\xi)} - 1 - \Delta_t(T,x,\xi) \big) \nu(d\xi) = 0
\\ &\text{for all $(T,x) \in \Xi$ with $T \geq t$}, \quad \text{$d \mathbb{P} \otimes dt$--almost surely on $\Omega \times \bbr_+$.}
\end{aligned}
\end{equation}
By Remark \ref{remark-a-capital}, Lemma \ref{lemma-A-loc} and Lebesgue's dominated convergence theorem, the left-hand side of (\ref{equiv-drift-2}) is even continuously differentiable in $T \geq -x \vee t$. Since the left-hand side of (\ref{equiv-drift-2}) evaluated at $T = -x \vee t$ vanishes, condition (\ref{equiv-drift-2}) is satisfied if and only if
\begin{equation}\label{drift-equi}
\begin{aligned}
&\alpha_t(T,x) = - \langle \sigma_t(T,x), \Sigma_t(T,x) \rangle_{L_2^0(\bbr)} - \int_E \delta_t(T,x,\xi) \big( e^{\Delta_t(T,x,\xi)} - 1 \big) \nu(d\xi)
\\ &\text{for all $(T,x) \in \Xi$ with $T \geq t$}, \quad \text{$d \mathbb{P} \otimes dt$--almost surely on $\Omega \times \bbr_+$,}
\end{aligned}
\end{equation}
and this is equivalent to (\ref{drift-cond}).
\end{proof}

\subsection{Proof of Theorem \ref{thm-drift-j-1}}\label{sec-proof-2}

In this appendix, we provide the proof of Theorem \ref{thm-drift-j-1}. We impose the following assumption:

\begin{assumption}\label{ass-a}
We suppose that the following conditions are satisfied:
\begin{enumerate}
\item $j_0$ is $\bcal(\Xi)$-measurable, $a$ and $b$ are $\mathcal{F}_{\infty} \otimes \mathcal{B}(\Theta)$-measurable, and $c$ is $\mathcal{F}_{\infty} \otimes \mathcal{B}(\Theta) \otimes \mathcal{E}$-measurable.

\item For all $(T,x) \in \Xi$ the processes $a(T,x)$ and $b(T,x)$ are optional, and $c(T,x)$ is predictable.

\item For each $x \in \bbr$ and every bounded Borel set $B \subset \Xi_x$ we have
$\int_B |j_0(T,x)| dT < \infty$.

\item For each $x \in \bbr$ and every bounded Borel set $B \subset \Theta_x$ there are random variables $X^{a} : \Omega \rightarrow \bbr$ and $X^{b}, X^{c} \in \lcal^2(\Omega)$ such that
\begin{equation}\label{cond-a}
\begin{aligned}
&|a_t(T,x)| \leq X^{a}, \quad \| b_t(T,x) \|_{L_2^0(\bbr)} \leq X^{b} \quad \text{and} \\ &\| c_t(T,x) \|_{\lcal_{\nu}^2(\bbr)} \leq X^{c} \quad \text{for all $(t,T) \in B$.}
\end{aligned}
\end{equation}

\item For all $t \in \bbr_+$ and $\xi \in E$ the mappings $(T,x) \mapsto a_t(T,x)$, $(T,x) \mapsto b_t(T,x)$, $(T,x) \mapsto c_t(T,x,\xi)$ are continuous.
\end{enumerate}
\end{assumption}

Condition (\ref{cond-a}) ensures that all subsequent stochastic integrals regarding $a$, $b$ and $c$ exist. In particular, for each $x \in \bbr$ and every bounded Borel set $B \subset \Theta_x$ we have
\begin{equation}\label{cond-a-Fubini}
\begin{aligned}
&\iint_B |a_t(T,x)| dt dT < \infty, \quad
\bbe \bigg[ \iint_B \| b_t(T,x) \|_{L_2^0(\bbr)}^2 dt dT \bigg] < \infty \quad \text{and}
\\ &\bbe \bigg[ \iint_B \| c_t(T,x) \|_{\lcal_{\nu}^2(\bbr)}^2 dt dT \bigg] < \infty.
\end{aligned}
\end{equation}
This ensures that we may apply the classical Fubini theorem and the stochastic Fubini theorems (Theorem 2.8 in \ci{Gawarecki-Mandrekar} and Theorem A.2 in \ci{SPDE}) later on.
Furthermore, Assumption \ref{ass-a} guarantees that the processes $\alpha$, $\sigma$ and $\delta$ defined in (\ref{def-alpha-a}) fulfill Assumption \ref{ass-alpha}. In addition, we require the following assumption:

\begin{assumption}\label{ass-Levy-measure-c}
We suppose that for each $n \in \bbn$ there exist a measurable mapping $\rho_n : E \rightarrow \bbr_+$ and a constant $\epsilon_n > 0$ such that (\ref{Levy-measure-1}) is satisfied and we have
\begin{align}\label{Levy-measure-c} 
&|c_t(T,x,\xi)| \leq \rho_n(\xi) \quad \text{for all $(t,T,x) \in \Theta_n$ and $\xi \in E$.}
\end{align}
\end{assumption}

\begin{lemma}\label{lemma-A-loc-next}
For each $n \in \bbn$ there exists a measurable mapping $\pi_n : E \rightarrow \bbr_+$ such that conditions (\ref{pi-integrable})--(\ref{e-Taylor}) are fulfilled and we have
\begin{align}\label{est-c-Taylor} 
&\big| c_t(T,x,\xi) \big( e^{\Delta_t(T,x,\xi)} - 1 \big) \big| \leq \pi_n(\xi) \quad \text{for all $(t,T,x) \in \Theta_n$ and $\xi \in E$,}
\\ \notag &\bigg| \bigg( \int_t^T c_t(u,T+x-u,\xi) du \bigg) \bigg( \int_{-x \vee t}^T c_t(u,x,\xi)du \bigg) e^{\Delta_t(T,x,\xi)} \bigg| \leq \pi_n(\xi) 
\\ \label{est-delta-Taylor} &\quad \text{for all $(t,T,x) \in \Theta_n$ and $\xi \in E$.}
\end{align}
\end{lemma}

\begin{proof}
By Assumption \ref{ass-Levy-measure-c} and Definition (\ref{def-alpha-a}) of $\delta$, for each $n \in \bbn$ there exist a measurable mapping $\rho_n : \Omega \times E \rightarrow \bbr_+$ and a constant $\epsilon_n > 0$ such that conditions (\ref{Levy-measure-1}), (\ref{Levy-measure-c}) are fulfilled and we have (\ref{Levy-measure-2a}). Proceeding as in the proof of Lemma \ref{lemma-A-loc} yields the desired estimates (\ref{est-c-Taylor}) and (\ref{est-delta-Taylor}).
\end{proof}

Now, we are ready to provide the proof of Theorem \ref{thm-drift-j-1}:

\begin{proof}[Proof of Theorem \ref{thm-drift-j-1}]
Let $x \in \bbr$ be arbitrary.
By Definition (\ref{def-gamma}) of $\gamma(x)$ and the dynamics (\ref{mortality-forces}) of $\mu(T,x)$ we have (\ref{spot-shifted-simple}), proving the first statement. In particular, for all $(T,x) \in \Xi$ we have
\begin{equation}\label{spot-shifted}
\begin{aligned}
\gamma_t(T+x-t) &= \mu_0(t,T+x-t) + \int_0^t \alpha_s(t,T+x-t) ds + \int_0^t \sigma_s(t,T+x-t) dW_s
\\ &\quad + \int_0^t \int_E \delta_s(t,T+x-t,\xi) (\mathfrak{p}(ds,d\xi) - \nu(d\xi)ds), \quad t \geq 0.
\end{aligned}
\end{equation}
Let $(T,x) \in \Xi$ be arbitrary. We also fix an arbitrary $t \in [0,T]$. By the dynamics (\ref{j-dynamics}) of $j(T,x)$ we have
\begin{equation}\label{j-integrated}
\begin{aligned}
&- \int_{t}^T j_t(u,T+x-u)du = - \int_{t}^T j_0(u,T+x-u) du - \int_{t}^T \int_{0}^t a_s(u,T+x-u)ds du
\\ &\quad - \int_{t}^T \int_{0}^t b_s(u,T+x-u)dW_s du 
\\ &\quad - \int_{t}^T \int_{0}^t \int_E c_s(u,T+x-u,\xi) (\mathfrak{p}(ds,d\xi) - \nu(d\xi)ds) du.
\end{aligned}
\end{equation}
We shall now consider the terms in (\ref{spot-shifted}) and (\ref{j-integrated}) separately. By Definition (\ref{def-mu-0-extend}) we have 
\begin{align*}
&\mu_0(t,T+x-t) - \int_{t}^T j_0(u,T+x-u) du 
\\ &= \gamma_0(T+x) - \int_0^t j_0(u,T+x-u)du - \int_t^T j_0(u,T+x-u)du 
\\ &= \gamma_0(T+x) - \int_0^T j_0(u,T+x-u)du = \mu_0(T,x).
\end{align*}
By (\ref{cond-a-Fubini}) we may apply the classical Fubini theorem for the following calculation. Incorporating Definition (\ref{def-alpha-a}) we obtain
\begin{align*}
&\int_0^t \alpha_s(t,T+x-t)ds - \int_t^T \int_0^t a_s(u,T+x-u)ds du
\\ &= \int_0^t \bigg( \alpha_s(t,T+x-t) - \int_t^T a_s(u,T+x-u)du \bigg) ds
\\ &= \int_0^t \bigg( - \int_s^t a_s(u,T+x-u)du - \int_t^T a_s(u,T+x-u)du \bigg) ds
\\ &= -\int_0^t \int_s^T a_s(u,T+x-u)du ds = \int_{0}^t \alpha_s(T,x) ds.
\end{align*}
Analogous calculations yield 
$$
\int_{0}^t \sigma_s(t,T+x-t) d W_s - \int_{t}^T \int_{0}^t b_s(u,T+x-u) dW_s du = \int_{0}^t \sigma_s(T,x) d W_s, 
$$
\begin{align*}
&\int_{0}^t \int_E \delta_s(u,T+x-u,\xi) (\mathfrak{p}(ds,d\xi) - \nu(d\xi)ds)  
\\ &\quad - \int_{t}^T \int_{0}^t \int_E c_s(u,T+x-u,\xi)(\mathfrak{p}(ds,d\xi) - \nu(d\xi)ds) du 
\\ &= \int_{0}^t \int_E \delta_s(T,x,\xi) (\mathfrak{p}(ds,d\xi) - \nu(d\xi)ds).
\end{align*}
Note that we may apply the respective stochastic Fubini theorems (\ci[Thm. 2.8]{Gawarecki-Mandrekar} and \ci[Thm. A.2]{SPDE}) due to condition (\ref{cond-a-Fubini}).
Consequently, by (\ref{spot-shifted}), (\ref{j-integrated}) and the previous identities we arrive at identity (\ref{identity-mu-gamma-j}), establishing the second statement.

Now, suppose that, in addition, Assumption \ref{ass-Levy-measure-c} is satisfied, and the drift $a$ is given by (\ref{a-defined}). By Definitions (\ref{def-alpha-a}), for all $\xi \in E$ we have
\begin{align}\label{zero-at-t-x}
\sigma_t(t,x) = 0 \quad \text{and} \quad \delta_t(t,x,\xi) = 0 \quad \text{for all $(t,x) \in \Xi$},
\end{align}
\begin{align}\label{b-directional-sigma}
b_t(T,x) = -(\partial_T - \partial_x)\sigma_t(T,x) \quad \text{and} \quad c_t(T,x,\xi) = -(\partial_T - \partial_x)\delta_t(T,x,\xi).
\end{align}
Let $(T,x) \in \Xi$ be arbitrary. We also fix an arbitrary $t \in [0,T]$.
In view of Definition (\ref{Sigma-tilde}) of $\Sigma$, if $-x \geq t$,   we obtain by (\ref{b-directional-sigma}), (\ref{cond-a}) and Lebesgue's dominated convergence theorem
\begin{align*}
&(\partial_T - \partial_x) \Sigma_t(T,x) = \partial_h \big|_{h=0} \Sigma_t(T+h,x-h) = -\partial_h \big|_{h=0} \int_{-(x-h)}^{T+h} \sigma_t(u,x-h)du
\\ &= -\partial_h \big|_{h=0} \int_{-x+h}^{T+h} \sigma_t(u,x-h)du = -\partial_h \big|_{h=0} \int_{-x}^{T} \sigma_t(u+h,x-h)du
\\ &= -\int_{-x}^{T} \partial_h \big|_{h=0} \sigma_t(u+h,x-h)du = -\int_{-x}^T (\partial_T - \partial_x) \sigma_t(u,x) du = \int_{-x}^T b_t(u,x)du;
\end{align*}
and, if $t \geq -x$, we observe by (\ref{zero-at-t-x}), (\ref{b-directional-sigma}) and (\ref{cond-a}) together with Lebesgue's dominated convergence theorem
\begin{align*}
&(\partial_T - \partial_x) \Sigma_t(T,x) = \partial_h \big|_{h=0} \Sigma_t(T+h,x-h) = -\partial_h \big|_{h=0} \int_{t}^{T+h} \sigma_t(u,x-h)du
\\ &= -\partial_h \big|_{h=0} \int_{t-h}^{T} \sigma_t(u+h,x-h)du = \sigma_t(t,x) - \int_t^T \partial_h \big|_{h=0} \sigma_t(u+h,x-h) du
\\ &= -\int_t^T (\partial_T - \partial_x) \sigma_t(u,x) du = \int_t^T b_t(u,x)du.
\end{align*}
Performing analogous calculations with $\Delta_t$, we deduce that for each $\xi \in E$ the directional derivative $\partial_T - \partial_x$ of the mappings $(T,x) \mapsto \Sigma_t(T,x)$ and $(T,x) \mapsto \Delta_t(T,x,\xi)$ in $(T,x)$ exists, and we have
\begin{align}\label{Sigma-directional} 
(\partial_T - \partial_x) \Sigma_t(T,x) = \int_{-x \vee t}^T b_t(u,x)du \quad \text{and} \quad (\partial_T - \partial_x) \Delta_t(T,x) = \int_{-x \vee t}^T c_t(u,x,\xi)du.
\end{align}
We define the stochastic process $\beta : \Omega \times \Theta \rightarrow \mathbb{R}$ as
\begin{align*}
\beta_t(T,x) := -\langle \sigma_t(T,x), \Sigma_t(T,x) \rangle_{L_2^0(\bbr)} - \int_E \delta_t(T,x,\xi) \big( e^{\Delta_t(T,x,\xi)} - 1 \big) \nu(d\xi).
\end{align*}
Taking into account (\ref{b-directional-sigma}), (\ref{Sigma-directional}),  Definition (\ref{def-alpha-a}) of $\sigma$ and $\delta$, and estimates (\ref{est-c-Taylor}), (\ref{est-delta-Taylor}) together with Lebesgue's dominated convergence theorem, the directional derivative $\partial_T - \partial_x$ of $\beta$ in $(T,x)$ exists, and we have
\begin{align*}
&-(\partial_T - \partial_x)\beta_t(T,x) 
\\ &= \langle \sigma_t(T,x),(\partial_T - \partial_x)\Sigma_t(T,x) \rangle_{L_2^0(\mathbb{R})}
+ \langle (\partial_T - \partial_x) \sigma_t(T,x),\Sigma_t(T,x) \rangle_{L_2^0(\mathbb{R})}
\\ &\quad +\int_E \delta_t(T,x,\xi) \cdot (\partial_T - \partial_x) \big( e^{\Delta_t(T,x,\xi)} - 1 \big) \nu(d\xi)
\\ &\quad + \int_E (\partial_T - \partial_x) \delta_t(T,x,\xi) \cdot \big( e^{\Delta_t(T,x,\xi)} - 1 \big) \nu(d\xi),
\end{align*}
and hence
\begin{align*}
-(\partial_T - \partial_x)\beta_t(T,x) &= \langle \sigma_t(T,x),(\partial_T - \partial_x)\Sigma_t(T,x) \rangle_{L_2^0(\mathbb{R})}
- \langle b_t(T,x),\Sigma_t(T,x) \rangle_{L_2^0(\mathbb{R})}
\\ &\quad + \int_E \delta_t(T,x,\xi) \cdot (\partial_T - \partial_x) \Delta_t(T,x,\xi) \cdot e^{\Delta_t(T,x,\xi)} \nu(d\xi)
\\ &\quad - \int_E c_t(T,x,\xi) \big( e^{\Delta_t(T,x,\xi)} - 1 \big) \nu(d\xi),
\end{align*}
so that we arrive at
\begin{align*}
-(\partial_T - \partial_x)\beta_t(T,x)  &= -\Big\langle \int_t^T b_t(u,T+x-u)du,\int_{-x \vee t}^T b_t(u,x)du \Big\rangle_{L_2^0(\mathbb{R})}
\\ &\quad + \Big\langle b_t(T,x),\int_{-x \vee t}^T \sigma_t(u,x)du \Big\rangle_{L_2^0(\mathbb{R})}
\\ &\quad -\int_E \bigg( \int_t^T c_t(u,T+x-u,\xi) du \bigg) \bigg( \int_{-x \vee t}^T c_t(u,x,\xi)du \bigg) 
\\ &\qquad\qquad \exp \bigg( -\int_{-x \vee t}^T \delta_t(u,x,\xi) du \bigg) \nu(d\xi)
\\ &\quad - \int_E c_t(T,x,\xi) \bigg[ \exp \bigg( -\int_{-x \vee t}^T \delta_t(u,x,\xi) du \bigg) - 1 \bigg] \nu(d\xi).
\end{align*}
Applying Definition (\ref{def-alpha-a}) of $\sigma$ and $\delta$ once again, we obtain by Definition~\eqref{a-defined}
\begin{align*}
-(\partial_T - \partial_x)\beta_t(T,x) = a_t(T,x),
\end{align*}
and hence, by Definition (\ref{def-alpha-a}) of $\alpha$ we deduce that
\begin{align*}
\alpha_t(T,x) = -\int_t^T a_t(u,T+x-u)du = \beta_t(T,x) - \beta_t(t,T+x-t) = \beta_t(T,x).
\end{align*}
Consequently, by Theorem \ref{thm-drift-HJM} the $\bbf$-forward mortality rates $\mu$ are consistent, completing the proof. 
\end{proof}

\subsection{Proof of Theorem \ref{thm-drift-HJM}}\label{sec-proof-3}

In this appendix, we provide the proof of Theorem \ref{thm-drift-HJM}. We impose the following assumption:

\begin{assumption}\label{ass-alpha-bar}
We suppose that the following conditions are satisfied:
\begin{enumerate}
\item The processes $\bar{\alpha}$ and $\bar{\sigma}$ are optional, and $\bar{\delta}$ is predictable.

\item For each bounded Borel set $B \subset \bbr_+$ there are random variables $X^{\bar{\alpha}} : \Omega \rightarrow \bbr$ and $X^{\bar{\sigma}}, X^{\bar{\delta}} \in \lcal^2(\Omega)$ such that
\begin{align*}
\| \bar{\alpha}_t \| \leq X^{\bar{\alpha}}, \quad \| \bar{\sigma}_t \|_{L_2^0(H)} \leq X^{\bar{\sigma}} \quad \text{and} \quad
\| \bar{\delta}_t \|_{\lcal_{\nu}^2(H)} \leq X^{\bar{\delta}} \quad \text{for all $t \in B$.}
\end{align*}
\end{enumerate}
\end{assumption}

Note that we may regard these processes as mappings $\bar{\alpha} : \Omega \times \bbr_+ \times \Xi \rightarrow \bbr$, $\bar{\sigma} : \Omega \times \bbr_+ \times \Xi \rightarrow L_2^0(\bbr)$ and $\bar{\delta} : \Omega \times \bbr_+ \times \Xi \times E \rightarrow \bbr$. We define the stochastic processes $\alpha : \Omega \times \Theta \rightarrow \mathbb{R}$, $\sigma : \Omega \times \Theta \rightarrow L_2^0(\mathbb{R})$ and $\delta : \Omega \times \Theta \times E \rightarrow \bbr$ as
\begin{align}\label{alpha-by-bar}
\alpha := \bar{\alpha} \circ \phi, \quad \sigma := \bar{\sigma} \circ \phi \quad \text{and} \quad \delta := \bar{\delta} \circ \phi.
\end{align}
By Assumption \ref{ass-alpha-bar} and estimate (\ref{point-eval-uniform}), the processes $\alpha$, $\sigma$ and $\delta$ also fulfill Assumption \ref{ass-alpha}. In addition, we require the following assumption:

\begin{assumption}\label{ass-Levy-measure-bar}
We suppose that for each $n \in \bbn$ there exist a measurable mapping $\rho_n : E \rightarrow \bbr_+$ and a constant $\epsilon_n > 0$ such that (\ref{Levy-measure-1}) is satisfied and we have
\begin{align*} 
|\bar{\delta}_t(s,y,\xi)| &\leq \rho_n(\xi) \quad \text{for all $(t,s,y) \in \phi(\Theta_n)$ and $\xi \in E$.}
\end{align*}
\end{assumption}

Note that Assumption \ref{ass-Levy-measure-bar} implies Assumption \ref{ass-Levy-measure}. Now, we are ready to provide the proof of Theorem \ref{thm-drift-HJM}:

\begin{proof}[Proof of Theorem \ref{thm-drift-HJM}] We define the transformed $\bbf$-forward mortality rates $\mu : \Omega \times \Theta \rightarrow \bbr$ as $\mu := \bar{\mu} \circ \phi$. For all $(t,T,x) \in \Theta$ and $(t,s,y) = \phi(t,T,x) \in \bbr_+ \times \Xi$ we have
\begin{align*}
\mu_0(T,x) = \mu_0(s+t,y-t) = S_t \mu_0(s,y),
\end{align*}
and, by Definition (\ref{alpha-by-bar}) we obtain
\begin{align*}
\int_0^t \alpha_s(T,x) ds &= \int_0^t \bar{\alpha}_s(T-s,x+s) du = \int_0^t S_{t-s} \bar{\alpha}_s(T-t,x+t) du 
\\ &= \int_0^t S_{t-s} \bar{\alpha}_s(s,y) ds.
\end{align*}
Analogous calculations for $\sigma$ and $\delta$ show that for all $(T,x)$ the $\bbf$-forward mortality rates $\mu(T,x)$ have the dynamics (\ref{mortality-forces}). Moreover, using (\ref{alpha-by-bar}), a straightforward calculation shows that conditions (\ref{drift-cond}) and (\ref{drift-cond-Musiela}) are equivalent. Therefore, applying Theorem \ref{thm-drift-HJM-mu} completes the proof.
\end{proof}

\subsection{Proof of Theorem \ref{thm-main-Musiela}}\label{sec-proof-4}

In this appendix, we provide the proof of Theorem \ref{thm-main-Musiela}. We impose the following assumption:

\begin{assumption}\label{ass-a-bar}
We suppose that the following conditions are satisfied:
\begin{enumerate}
\item The processes $\bar{a}$ and $\bar{b}$ are optional, and $\bar{c}$ is predictable.

\item For each bounded Borel set $B \subset \bbr_+$ there are random variables $X^{\bar{a}} : \Omega \rightarrow \bbr$ and $X^{\bar{b}}, X^{\bar{c}} \in \lcal^2(\Omega)$ such that
\begin{align*}
\| \bar{a}_t \| \leq X^{\bar{a}}, \quad \| \bar{b}_t \|_{L_2^0(H)} \leq X^{\bar{b}} \quad \text{and} \quad
\| \bar{c}_t \|_{\lcal_{\nu}^2(H)} \leq X^{\bar{c}} \quad \text{for all $t \in B$.}
\end{align*}
\end{enumerate}
\end{assumption}

Note that we may also regard these processes as mappings $\bar{a} : \Omega \times \bbr_+ \times \Xi \rightarrow \bbr$, $\bar{b} : \Omega \times \bbr_+ \times \Xi \rightarrow L_2^0(\bbr)$ and $\bar{c} : \Omega \times \bbr_+ \times \Xi \times E \rightarrow \bbr$. We define the stochastic processes $a : \Omega \times \Theta \rightarrow \mathbb{R}$, $b : \Omega \times \Theta \rightarrow L_2^0(\mathbb{R})$ and $c : \Omega \times \Theta \times E \rightarrow \bbr$ as
\begin{align}\label{a-by-bar}
a := \bar{a} \circ \phi, \quad b := \bar{b} \circ \phi \quad \text{and} \quad c := \bar{c} \circ \phi.
\end{align}
By Assumption \ref{ass-a-bar} and estimate (\ref{point-eval-uniform}), the processes $a$, $b$ and $c$ also fulfill Assumption \ref{ass-a}.
In addition, we require the following assumption:

\begin{assumption}\label{ass-Levy-measure-c-bar}
We suppose that for each $n \in \bbn$ there exist a measurable mapping $\rho_n : E \rightarrow \bbr_+$ and a constant $\epsilon_n > 0$ such that (\ref{Levy-measure-1}) is satisfied and we have
\begin{align*}
|\bar{c}_t(s,y,\xi)| \leq \rho_n(\xi) \quad \text{for all $(t,s,y) \in \phi(\Theta_n)$ and $\xi \in E$.}
\end{align*}
\end{assumption}

Now, we are ready to provide the proof of Theorem \ref{thm-main-Musiela}:

\begin{proof}[Proof of Theorem \ref{thm-main-Musiela}] 
The identity (\ref{spot-shifted-bar}) follows from the variation of constants formula (\ref{var-const}) and the Definition (\ref{def-gamma-bar}) of the $\bbf$-spot mortality rates $\bar{\gamma}$.

We define the transformed $\bbf$-forward mortality rates $\mu : \Omega \times \Theta \rightarrow \bbr$ as $\mu := \bar{\mu} \circ \phi$ and the transformed $\bbf$-forward mortality improvements $j : \Omega \times \Theta \rightarrow \bbr$ as $j := \bar{j} \circ \phi$. Using Definitions (\ref{alpha-by-bar}) and (\ref{a-by-bar}), analogous calculations as in the proof of Theorem \ref{thm-drift-HJM} show that for all $(T,x)$ the $\bbf$-forward mortality rates $\mu(T,x)$ have the dynamics (\ref{mortality-forces}) and the $\bbf$-forward mortality rates $j(T,x)$ have the dynamics (\ref{j-dynamics}). 

Now, identity (\ref{identity-mu-gamma-j-bar}) follows from identity (\ref{identity-mu-gamma-j}) of Theorem \ref{thm-drift-j-1}. Furthermore, if 
Assumption \ref{ass-Levy-measure-c-bar} is satisfied, then condition (\ref{a-bar-defined}) and Definitions (\ref{a-by-bar}) imply that condition (\ref{a-defined}) holds true. Applying Theorem \ref{thm-drift-j-1} yields that the $\bbf$-forward mortality rates $\bar{\mu}$ are consistent, which concludes the proof.
\end{proof}
\end{appendix}

\bibliographystyle{jmr}
\bibliography{bibtex}   

%
%

\end{document}